\documentclass{amsart}
\usepackage{amscd,amssymb,euscript}
\usepackage[all]{xy}

\renewcommand\phi{\varphi}

\newcommand\gm{\mathop{\mathbb{G}\mathrm{m}}}
\DeclareMathOperator\detRG{{detR\Gamma}}
\DeclareMathOperator\Lie{{Lie}}
\DeclareMathOperator\Pic{{Pic}}
\DeclareMathOperator\supp{{supp}}
\DeclareMathOperator\spec{{Spec}}
\DeclareMathOperator\codim{{codim}}
\DeclareMathOperator\cPic{{\mathcal{P}\!\mathit{ic}}}
\DeclareMathOperator{\Id}{Id}
\DeclareMathOperator{\rk}{rk}
\DeclareMathOperator{\Ad}{Ad}
\DeclareMathOperator{\ad}{ad}
\DeclareMathOperator{\St}{St}
\DeclareMathOperator{\HOM}{\mathcal{H}\mathit{om}}
\DeclareMathOperator{\Hom}{Hom}
\DeclareMathOperator{\Mor}{Mor}
\DeclareMathOperator{\Aut}{Aut}
\DeclareMathOperator{\Iso}{Iso}

\newcounter{noindnum}[subsection]
\renewcommand{\thenoindnum}{\roman{noindnum}}
\newcommand{\noindstep}{\refstepcounter{noindnum}{\rm(}\thenoindnum\/{\rm)} }
\newcommand{\stepzero}{\setcounter{noindnum}{0}}

\newcounter{noindnumalph}[subsection]
\renewcommand{\thenoindnumalph}{\alph{noindnumalph}}
\newcommand{\noindstepalph}{\refstepcounter{noindnumalph}{\rm(}\thenoindnumalph\/{\rm)} }

\theoremstyle{plain}
\newtheorem{theorem}{Theorem}
\newtheorem*{theorem*}{Theorem}
\newtheorem{proposition}{Proposition}[section]
\newtheorem{lemma}[proposition]{Lemma}
\newtheorem{corollary}[proposition]{Corollary}

\theoremstyle{definition}
\newtheorem{definition}[proposition]{Definition}

\theoremstyle{remark}
\newtheorem{remark}[proposition]{Remark}
\newtheorem*{remark*}{Remark}
\newtheorem{remarks}[proposition]{Remarks}
\newtheorem*{remarks*}{Remarks}
\newtheorem{example}[proposition]{Example}
\newtheorem*{conventions*}{Conventions}

\newcommand\tA{{\widetilde{A}}}
\newcommand\tB{{\widetilde{B}}}
\newcommand\tL{{\widetilde{L}}}

\newcommand\cC{{\mathcal{C}}}
\newcommand\cF{{\mathcal{F}}}
\newcommand\cE{{\mathcal{E}}}
\newcommand\cJ{{\mathcal{J}}}
\newcommand\cL{{\mathcal{L}}}

\newcommand\cM{{\mathcal{M}}}
\newcommand\cO{{\mathcal{O}}}
\newcommand\cT{{\mathcal{T}}}
\newcommand\cX{{\mathcal{X}}}

\newcommand\Z{\mathbb Z}
\newcommand\C{\mathbb C}

\newcommand\fa{\mathfrak a}
\newcommand\fg{\mathfrak g}
\newcommand\fm{\mathfrak m}

\newcommand{\Bun}{\mathcal{B}un}

\newcommand{\Hitch}{\mathcal{H}itch}

\title[Partial Fourier--Mukai transform]{Partial
Fourier--Mukai transform for integrable systems with applications to Hitchin fibration}

\author{Dima Arinkin}
\email{arinkin@math.wisc.edu}
\address{Department of mathematics, University of Wisconsin,
Madison, Wisconsin}

\author{Roman Fedorov}
\email{fedorov@math.ksu.edu}
\address{Department of Mathematics, Kansas State University, Manhattan, Kansas}
\address{Max Planck Institute for Mathematics, Bonn, Germany}

\begin{document}

\begin{abstract}
Let $X$ be an abelian scheme over a scheme $B$. The Fourier--Mukai transform gives an equivalence between the derived category of $X$ and the derived category of the dual abelian scheme. We partially extend this to certain schemes $X$ over $B$ (which we call degenerate abelian schemes) whose generic fiber is an abelian variety, while special fibers are singular.

Our main result provides a fully faithful functor from a twist of the derived category of $\Pic^\tau(X/B)$ to the derived category of $X$. Here $\Pic^\tau(X/B)$ is the algebraic space classifying fiberwise numerically trivial line bundles.

Next, we show that every algebraically integrable system gives rise to a~degenerate abelian scheme and discuss applications to Hitchin systems.
\end{abstract}

\maketitle

\section{Introduction}
Let $X$ be an abelian variety and $\check X$ be the dual abelian variety. The famous Fourier--Mukai transform~\cite{Mukai} is an equivalence of derived categories
\begin{equation}\label{eq:FM}
    D(X)\xrightarrow{\simeq} D(\check X).
\end{equation}
A similar equivalence of categories can be constructed for a family of abelian varieties, that is, for an abelian scheme. In this paper, we consider \emph{degenerate abelian schemes\/}, that is, families whose generic fiber is a torsor over an abelian variety, while some special fibers are singular projective varieties (see Definition~\ref{def:DAS} below).

\subsection{} We would like to formulate our main result in a somewhat simplified situation. Let $k$ be a field of characteristic zero; let $\pi_X:X\to B$ be a flat projective morphism of smooth connected $k$-schemes. We assume that the geometric fibers of  $\pi_X$ are integral. Denote by $A$ the locus where the morphism $\pi_X$ is smooth; set $\pi_A:=\pi_X|_A$. We assume that $A$ is given a structure of a commutative group scheme over $B$ and the action of $A$ on itself by translations extends to an action of $A$ on $X$
\[
    \mu:A\times_BX\to X.
\]
Note that $\pi_X$ is smooth over an open subset $B_0$ of $B$, thus $A$ coincides with $X$ over~$B_0$, and we see that $\pi_A^{-1}(B_0)$ is an abelian scheme over $B_0$.

Let $\Pic^\tau(X/B)$ be the algebraic space classifying fiberwise numerically trivial line bundles over $\pi_X:X\to B$. We will see that it is separated and of finite type over $B$. Moreover, since the zero section of $A$ gives a section of $\pi_X$, $\Pic^\tau(X/B)$ is a fine moduli space, that is, it admits a universal line bundle
\[
    L\to\Pic^\tau(X/B)\times_BX.
\]
Denote the canonical zero section morphism by
\[
\iota':B\to\Pic^\tau(X/B).
\]
Given a point $b\in B$, the fiber $A_b$ of $A$ over $b$ is an algebraic group. Let $\delta(A_b)$ denote the dimension of the maximal connected affine subgroup in $A_b$ (existing by the Chevalley Theorem). Our first main result is (cf.~Theorem~\ref{th:PartialFourier} and Corollary~\ref{co:PartialFourier})
\begin{theorem*}
Suppose that the group scheme $\pi_A:A\to B$ satisfies the following condition for all $k\ge0$
\begin{equation}\label{eq:condition_intro}
    \codim\overline{\{b\in B: \delta(A_b)\ge k\}}\ge k,
\end{equation}
where the bar denotes the Zariski closure. Then
\begin{itemize}
\item the morphism $\Pic^\tau(X/B)\to B$ is smooth;

\item we have
\[
    Rp_{1,*}L\simeq\iota'_*\ell[-g],
\]
where $\ell$ is a certain line bundle over $B$, and $[-g]$ is the cohomological shift by the relative dimension $g$ of $\pi_X$;

\item the Fourier--Mukai transform with kernel $L$
\[
    D(\Pic^\tau(X/B))\to D(X): \cF\mapsto Rp_{2,*}(L\otimes p^*_1\cF)
\]
is fully faithful.
\end{itemize}
\end{theorem*}

This theorem is a generalization of the Fourier--Mukai transform for abelian schemes.

Below we relax the hypotheses of the theorem. In particular, we assume that~$A$ is a $B$-group scheme acting on~$X$ but~$A$ is not necessarily a subscheme of $X$. In this situation, the morphism $\pi_X$ might not have a section, and we must work with a stack $\cPic^\tau(X/B)$ of fiberwise numerically trivial line bundles, instead of the corresponding moduli space $\Pic^\tau(X/B)$. Note that $\cPic^\tau(X/B)$ is a $\gm$-gerbe over $\Pic^\tau(X/B)$.

Optimistically, one may hope that there exists a dual degenerate abelian scheme $\check X\to B$ and a Fourier--Mukai functor giving an equivalence of categories
\[
    FM:D(\check X)\xrightarrow{\simeq}D(X),
\]
at least under some additional assumptions. If such $\check X$ exists, it is natural to expect that
there is an open embedding $j:\Pic^\tau(X/B)\hookrightarrow\check X$ such that the partial Fourier--Mukai transform from the above theorem is equal to $FM\circ Rj_*$. Currently, we do not know how to construct $\check X$, except in some special cases, see~\cite{ArinkinAutoduality,BridgelandMaciocia}.

\subsection{Integrable systems}
It turns out that the conditions of the above theorem are satisfied for any \emph{algebraically integrable system}. As above,
assume that $k$ is a field of characteristic zero, let $X$ and $B$ be smooth connected $k$-schemes, and let $\pi_X:X\to B$ be a flat projective morphism with integral geometric fibers. In addition, assume that $X$ is symplectic and $\pi_X$ is a Lagrangian fibration, that is, $\dim X=2\dim B$ and the restriction of the symplectic form to the smooth locus of any fiber of $\pi_X$ is zero. For the purpose of the introduction, we also assume that we are given a section of $\pi_X$. We will see in Theorem~\ref{th:IntegrableSystems} that in this case the smooth locus $A$ of $\pi_X$ is canonically a group scheme over $B$, and that the action of $A$ on itself by translations extends to an action of $A$ on $X$. This statement is a~version of the Liouville theorem. We will also see that the group scheme satisfies the dimensional estimates~\eqref{eq:condition_intro} automatically; this goes back to~\cite{NgoBonn}. Thus, the above theorem applies.

Important examples of algebraically integrable systems are the Hitchin integrable systems. In this case our results form a part of \emph{the Langlands duality for the Hitchin systems}, see Section~\ref{sect:Hitchin}; this application is our primary motivation.  While there are many results on the duality for Hitchin systems for $GL(n)$ (see below), very little is known for other reductive groups, beyond the case of smooth fibers settled in~\cite{DonagiPantev}. This paper grew from our attempts to extend the results of \cite{DonagiPantev}. Note that the Langlands duality for Hitchin systems is a limit of the categorical Langlands duality.

\subsection{Compactified Jacobians}
Assume that $C$ is a smooth projective curve, then $X:=\Pic^0(C)$ is an abelian variety. Its dual variety is isomorphic to $X$; so the equivalence of categories~\eqref{eq:FM} is not surprising. However, the Fourier--Mukai transform is still highly non-trivial; for example, it takes skyscraper sheaves to line bundles. There has been a lot of progress on extending the Fourier--Mukai transform to compactified Jacobians of singular curves; see~\cite{ArinkinJacobians,ArinkinAutoduality,MeloETalJacobians,MeloETalJacobiansII} and references therein. The present paper extends the techniques of \cite{ArinkinJacobians}, and may be viewed as its continuation.

From the point of view of the Hitchin systems, compactified Jacobians of curves with planar singularities are fibers of the Hitchin systems for $GL(n)$. The results of the present paper are applicable to the Hitchin systems for other reductive groups.

\subsection{Conventions}
For a scheme $B$ the notation $b\in B$ means that $b$ is a schematic point of $B$. We denote by $k(b)$ its residue field. For an integral scheme $Z$ we denote by $k(Z)$ its field of rational functions. If $Z$ is a scheme over $B$ and $b\in B$, we put $Z_b:=Z\times_B~\spec k(b)$. If $S$ is a $B$-scheme, we sometimes denote $Z\times_BS$ by $Z_S$. We usually denote the projections from a product of two schemes onto the two factors by $p_1$ and $p_2$.

In different publications, the term `projective morphism' has slightly different meanings. When we say `$\pi_X:X\to B$ is a projective morphism', we mean `locally in the Zariski topology of $B$, the morphism factorizes as a composition of the closed embedding $X\to{\mathbb P}^n\times B$ and the projection onto $B$', where $n$ may vary.

If $\ell$ is a line bundle, $\ell^\vee$ denotes the dual line bundle.

The symbol `$\simeq$' means that two objects are isomorphic; we use the equality `$=$' to emphasize that the isomorphism is canonical.

We often work over a fixed field $k$, in which case the schemes are assumed to be locally of finite type over $k$. We do not assume that the field is algebraically closed, partly because we systematically work with generic points.

If $Z\to T$ is a smooth morphism of schemes, we denote by $\cT(Z/T)$ the relative tangent sheaf; this is a locally free sheaf on $Z$. If we work over a fixed field $k$, we denote $\cT(Z/\spec k)$ simply by $\cT(Z)$.

If $\cF_1$ and $\cF_2$ are sheaves of $\cO_Z$-modules on an algebraic space $Z$, we denote by $\HOM(\cF_1,\cF_2)$ the sheaf of homomorphisms, while $\Hom(\cF_1,\cF_2):=\Gamma(\HOM(\cF_1,\cF_2))$ is the group of (global) homomorphisms.

\subsection{Acknowledgments} The first author benefited from talks with Vladimir Drinfeld, Bao Ch{\^a}u Ng{\^o}, Tony Pantev, and Vivek Shende. He also wants to thank Preena Samuel for her remark. The second author gave talks on the subject at University of Strasbourg, Mathematical Research Institute of Oberwolfach, University of Cologne, University of Mainz, Swiss Federal Institute of Technology Lausanne, Max Planck Institute for Mathematics, the University of Sheffield, and at Warwick EPSRC Symposium. He wants to thank the organizers and the participants of the talks for the attention and the questions. He especially wants to thank H\'el\`ene Esnault and Tom Bridgeland for their interest in the subject.

The first author is partially supported by NSF grant DMS-1101558. The second author is partially supported by NSF grant DMS-1406532.

\section{Main results}\label{MainResults}
Fix a field $k$ of characteristic zero and a smooth $k$-scheme $B$ locally of finite type. Assume also that $B$ has pure dimension. For an algebraic group $G$ over a field of characteristic zero, let $\delta(G)$ denote the dimension of the maximal connected affine subgroup in $G$ (which exists by the Chevalley Theorem, see~\cite{Chevalley,MilneChevalley,BrianConradChevalley}).

\subsection{Degenerate abelian schemes}
\begin{definition}\label{def:DAS} A \emph{degenerate abelian scheme} over $B$ is a triple
\[
    (\pi_X:X\to B,\pi_A:A\to B,\mu:A\times_B X\to X),
\]
where $\pi_X:X\to B$ is a scheme over $B$, $\pi_A:A\to B$ is a group scheme over $B$, and $\mu:A\times_B X\to X$ is an action of $A$ on $X$ such that the following conditions are satisfied\\
\stepzero\noindstep $\pi_X$ is a flat Gorenstein projective morphism whose geometric fibers are non-empty, reduced, connected, and of pure dimension;\\
\noindstep\label{DAS:RelDualLB} the relative dualizing line bundle $\omega_{X/B}$ (which exists because $\pi_X$ is Gorenstein) is isomorphic to the pullback of a line bundle over $B$;\\
\noindstep $A$ is a smooth commutative $B$-group scheme of finite type;\\
\noindstep $A$ satisfies the dimensional estimates
\begin{equation}\label{eq:condition}
    \codim\overline{\{b\in B: \delta(A_b)\ge k\}}\ge k\qquad(k\ge0).
\end{equation}
\noindstep\label{DAS:gentorsor} over an open dense subset of $B$,
$X$ is an $A$-torsor;\\
\noindstep\label{DAS:FiniteStab} for every schematic point $b\in B$, there is a point $x\in X_b$ such that the stabilizer of $x$ in $A\times_Bx$ is finite.
\end{definition}
\begin{remarks}
\stepzero\noindstepalph Precisely, condition~\eqref{DAS:gentorsor} means that there is an \'etale morphism $B'\to B$ with a dense image such that $A\times_BB'$ and $X\times_BB'$ are isomorphic as $B'$-schemes with action of $A\times_BB'$.

\noindstepalph Condition~\eqref{DAS:FiniteStab} means that the fiber of the orbital morphism $A\times_Bx\to X$ over $x$ is finite.

\noindstepalph
Condition~\eqref{DAS:gentorsor} implies that over a dense open subset $B_0\subset B$, the group scheme~$A$ is an abelian group scheme.
Thus over $B_0$, the scheme $X$ is a torsor over an abelian scheme, which implies that its relative dualizing line bundle is isomorphic to the pullback of a line bundle over $B_0$. Hence by Propositions~\ref{pr:zerosection} and~\ref{pp:Cartesian}\eqref{pp:Cart2}, condition~\eqref{DAS:RelDualLB} is equivalent to the following condition\\
(\ref{DAS:RelDualLB}') The relative dualizing line bundle $\omega_{X/B}$ is fiberwise numerically trivial (see Definition~\ref{def:FibNumTriv}).\\
Also, one can show that \eqref{DAS:RelDualLB} is implied by the following condition\\
(\ref{DAS:RelDualLB}'') The locus of $b\in B$ such that the fiber $X_b$ is geometrically reducible has codimension at least two in $B$.

\noindstepalph
A somewhat non-intuitive part of the definition is dimensional estimates~\eqref{eq:condition}. These estimates play a crucial role in
Ng\^o's proof of Fundamental Lemma, see~\cite[Sect.~7.1.5]{NgoIHES}. They also appear in the computation of cohomology groups of line
bundles over compactified Jacobians, see~\cite{ArinkinJacobians}.

\noindstepalph
Note that by Proposition~\ref{pp:delta semicontinuous} the set
\[
    \{b\in B:\delta(A_b)\ge k\}
\]
is closed. Thus, Zariski closure in~\eqref{eq:condition} is not necessary. We can also re-write~\eqref{eq:condition} as
\[
    \forall b\in B:\delta(A_b)\le\codim\overline b.
\]
\end{remarks}

Let $\cPic^\tau(X/B)\to B$ be the stack classifying fiberwise numerically trivial line bundles over $X\to B$. As we will see in Section~\ref{sect:relativePic}, it is an algebraic stack of finite type over $k$. Moreover, it is
a $\gm$-gerbe over an algebraic space $\Pic^\tau(X/B)$ (the coarse moduli space of fiberwise numerically trivial line bundles).

By Proposition~\ref{pp:Cartesian}, the zero section provides a natural Cartesian diagram
\begin{equation}\label{eq:diagram}
\begin{CD}
B\times {\mathrm B}(\gm) @>\iota>>\cPic^\tau(X/B)\\
@VVV @VVV\\
B @>\iota'>>\Pic^\tau(X/B),
\end{CD}
\end{equation}
where ${\mathrm B}(\gm)=pt/\gm$ is the classifying stack of the group $\gm$. See Section~\ref{sect:relativePic} for details and precise definitions.

We may identify quasicoherent sheaves on $B\times {\mathrm B}(\gm)$ with quasicoherent sheaves on $B$ equipped with a linear action of $\gm$. If $\ell$ is a line bundle over $B$ and $i$ is an~integer, we denote by $\ell^{(i)}$ the line bundle over $B\times {\mathrm B}(\gm)$ corresponding to $\ell$ with $\gm$ acting according to the character $\lambda\mapsto\lambda^i$.

Our first main result is
\begin{theorem}\label{th:PartialFourier}
Assume that $k$ is a field of characteristic zero, $B$ is a smooth scheme of pure dimension locally of finite type over $k$, and
\[
    (\pi_X:X\to B,\pi_A:A\to B,\mu:A\times_B X\to X)
\]
is a degenerate abelian scheme. Then
\begin{enumerate}
\item\label{th:PartFourier1} the morphisms $\cPic^\tau(X/B)\to B$ and $\Pic^\tau(X/B)\to B$ are smooth;

\item\label{th:PartFourier2} let $\cL$ be the universal line bundle over $\cPic^\tau(X/B)\times_B X$. Then there is a~line bundle $\ell$ over $B$ such that
\[
    Rp_{1,*}\cL\simeq\iota_*\ell^{(1)}[-g],
\]
where $[-g]$ is the cohomological shift by the relative dimension $g$ of $\pi_X$.
\end{enumerate}
\end{theorem}

This theorem will be proved in Section~\ref{ProofOfMainThm}.

\begin{remarks}
    \stepzero\noindstep It follows from the proof that $\ell$ is isomorphic to $\wedge^g\fa$, where $\fa$ is the~Lie algebra of $\pi_A:A\to B$.

    \noindstep In the proof of the theorem we will see that there is an open substack $\cPic^0(X/B)$ of $\cPic^\tau(X/B)$ classifying deformationally trivial line bundles. For an~abe\-lian scheme $\pi_X:X\to B$ we have $\cPic^0(X/B)=\cPic^\tau(X/B)$. We do not know if this equality remains valid for degenerate abelian schemes.
\end{remarks}

A standard argument interprets the second statement of Theorem~\ref{th:PartialFourier} as full faithfulness of a Fourier--Mukai transform.
Specifically, let $D(\cPic^\tau(X/B))$ be the quasicoherent derived category of $\cPic^\tau(X/B)$. Since $\cPic^\tau(X/B)$ is a $\gm$-gerbe, the~category decomposes by characters of $\gm$. Denote by
\[
    D^{(-1)}(\cPic^\tau(X/B))\subset D(\cPic^\tau(X/B))
\]
the full subcategory of objects $\cF\in D(\cPic^\tau(X/B))$ such that $\gm$ acts on the cohomology $H^\bullet(\cF)$ via the character
$\lambda\mapsto\lambda^{-1}$.

\begin{corollary}\label{co:PartialFourier}
In the assumptions of Theorem~\ref{th:PartialFourier}, the Fourier--Mukai functor with kernel $\cL$
\[
    D^{(-1)}(\cPic^\tau(X/B))\to D(X):\cF\mapsto Rp_{2,*}(\cL\otimes p^*_1\cF)
\]
is fully faithful.
\end{corollary}

\begin{remark}\label{rm:section} Suppose that $\pi_X:X\to B$ admits a section $\zeta_X:B\to X$. Then $\Pic^\tau(X/B)$ is a fine moduli space, that is, there is a universal line bundle $L$ over $\Pic^\tau(X/B)\times_BX$; this line bundle $L$ gives a trivialization of the $\gm$-gerbe $\cPic^\tau(X/B)\to\Pic^\tau(X/B)$, see Section~\ref{sect:relativePic}. In this case Theorem~\ref{th:PartialFourier}\eqref{th:PartFourier2} can be stated simply as
\[
    Rp_{1,*}L\simeq\iota'_*\ell[-g],
\]
for a line bundle $\ell$ over $B$. (Here $\iota'$ is the bottom arrow in~\eqref{eq:diagram}.) Similarly, the~Fourier--Mukai functor $D(\Pic^\tau(X/B))\to D(X)$ given by $L$ is fully faithful.
\end{remark}

\subsection{Example: integrable systems}
It turns out that any algebraically completely integrable system gives rise to a degenerate abelian scheme.
Thus the above theorem and its corollary apply in this situation. Let us give the definitions. As before, let $k$ be a field of characteristic zero.

\begin{definition}\label{def:IS} Let $B$ be a smooth $k$-scheme of pure dimension $g$. An \emph{algebraically completely integrable system\/} (or just an \emph{integrable system} for brevity) over $B$ is a~pair $(\pi_X:X\to B,\omega)$, where $\pi_X:X\to B$ is a~scheme over $B$, $\omega\in H^0(X,\wedge^2\Omega_X)$ is a 2-form such that the following conditions are satisfied\\
\stepzero
\noindstep $X$ is smooth over $k$;\\
\noindstep $\pi_X$ is a flat projective morphism whose geometric fibers are non-empty and integral of dimension $g$;\\
\noindstep $\omega$ is a symplectic form on $X$;\\
\noindstep the smooth loci of the fibers of $\pi_X$ are Lagrangian (that is, the restriction of~$\omega$ to the smooth locus of any fiber is zero).
\end{definition}

Note that $\dim X=2g$ since $\pi_X$ is a flat morphism of relative dimension $g$ and $\dim B=g$. Denote by $X^{sm}$ the smooth locus of the projection $X\to B$. Our second main result is

\begin{theorem}\label{th:IntegrableSystems}
Let $(\pi_X:X\to B,\omega)$ be an integrable system. Then
\begin{enumerate}
    \item there is a canonically defined smooth commutative group scheme $\pi_A:A\to B$ and an action $\mu:A\times_BX\to X$ such that
    $X^{sm}$ is a torsor over $A$;
    \item the triple $(\pi_X,\pi_A,\mu)$ is a degenerate abelian scheme in the sense of Definition~\ref{def:DAS}.
\end{enumerate}
\end{theorem}
This theorem will be proved in Section~\ref{IntegrableSystems}. The first statement is an algebraic version of the Liouville theorem. Surprisingly, we were unable to find it in the literature. The main part of the second statement is that $A$ satisfies the dimensional estimates~\eqref{eq:condition}; this goes back to Ng\^o~\cite{NgoBonn}. Combining this theorem with Theorem~\ref{th:PartialFourier} and Corollary~\ref{co:PartialFourier}, we get
\begin{corollary}\label{cor:IntegrableSystems}
Let $(\pi_X:X\to B,\omega)$ be an integrable system. Then
\begin{enumerate}
\item the morphisms $\cPic^\tau(X/B)\to B$ and $\Pic^\tau(X/B)\to B$ are smooth;

\item let $\cL$ be the universal line bundle over $\cPic^\tau(X/B)\times_B X$. Then there is a~line bundle $\ell$ over $B$ such that
\[
    Rp_{1,*}\cL\simeq\iota_*\ell^{(1)}[-g];
\]
\item the Fourier--Mukai functor with kernel $\cL$
\[
    D^{(-1)}(\cPic^\tau(X/B))\to D(X):\cF\mapsto Rp_{2,*}(\cL\otimes p^*_1\cF)
\]
is fully faithful.
\end{enumerate}
\end{corollary}

\subsection{Hitchin systems}\label{sect:Hitchin}
The Hitchin systems provide examples of integrable systems particularly important for the geometric Langlands Program. Below we sketch the applications of our results to the Hitchin systems. We plan to return to this subject in detail in the future.

We assume for simplicity that $k$ is the field of complex numbers. Let $C$ be a~smooth projective curve over $k$ of genus at least two; let $G$ be a simple $k$-group. For simplicity we assume that $G$ is simply-connected; denote by $Z(G)$ its center. Denote by $\Omega_C$ the invertible sheaf of differentials on $C$. Let $\Bun_G=\Bun_G(C)$ be the moduli stack of principal $G$-bundles over $C$. Consider its cotangent stack $T^*\Bun_G$. It follows from deformation theory that $T^*\Bun_G$ parameterizes Higgs pairs, that is, pairs $(\cE,\Phi)$, where $\cE$ is a principal $G$-bundle over $C$, $\Phi$ is a section of $\ad\cE\otimes\Omega_C$ ($\ad\cE$ is the adjoint vector bundle).

Let $\fg$ be the Lie algebra of $G$. Consider \emph{the Hitchin fibration}
\begin{equation*}
    p:T^*\Bun_G\to\Hitch_G:=\bigoplus_{i=1}^l H^0(C,\Omega_C^{\otimes d_i}),
\end{equation*}
where $l$ is the rank of $\fg$, and $d_1,\ldots,d_l$ are the exponents of $\fg$. Recall the definition of  $p$. The algebra $\C[\fg]^G$ of $\Ad$-invariant functions on $\fg$ is freely generated by certain homogeneous polynomials $P_1$,\ldots,$P_l$ such that $\deg P_i=d_i$ (this can be taken as the definition of exponents $d_i$). Take $(\cE,\Phi)\in T^*\Bun_G$. We can trivialize $\cE$ locally over $C$ (this can be done locally in the Zariski topology by Steinberg's theorem; on the other hand, the \'etale topology would suffice for our purposes). Then $\Phi$ becomes a $\fg$-valued 1-form. Applying $P_i$ to this 1-form, we get a local section of $\Omega_C^{\otimes d_i}$. However, since $P_i$ is $\Ad$-invariant, this section does not depend on the trivialization of $\cE$. We see that the local sections of $\Omega_C^{\otimes d_i}$ thus obtained patch into a well-defined global section of $\Omega_C^{\otimes d_i}$. Performing this procedure for $i=1,\ldots,l$, we get the Hitchin map.

By~\cite[Thm.~2.2.4]{BeilinsonDrinfeldHitchin} $p$ is flat and its fibers are Lagrangian (see also the references therein).

Let $g\in G$ be a semisimple element. The centralizer $Z_G(g)$ of $g$ in $G$ is a reductive subgroup. It is easy to show that, up to conjugation, we obtain only finitely many reductive subgroups in this way. For every semisimple $g$, we have a finite morphism $\Hitch_{Z_G(g)}\to\Hitch_G$. This morphism only depends on the conjugacy class of $Z_G(g)$ in $G$, thus there are only finitely many different morphisms. Define $B$ to be the (open) complement in $\Hitch_G$ of images of $\Hitch_{Z_G(g)}$, where $g$ ranges over non-central semisimple elements of $G$. Set
\[
    \cX:=T^*\Bun_G\times_{\Hitch_G}B.
\]
One checks that $\cX$ is a $Z(G)$-gerbe over a smooth scheme $X$. Moreover, one can show that the morphism $\pi_X:X\to B$ is projective and has integral fibers. It is clear that $X$ carries a canonical symplectic form $\omega$ and the pair $(\pi_X:X\to B,\omega)$ is an integrable system. Thus we can apply Theorem~\ref{th:IntegrableSystems} and Corollary~\ref{cor:IntegrableSystems}. We see that there is a fully faithful embedding
\[
    D^{(-1)}(\cPic^\tau(X/B))\hookrightarrow D(X).
\]
The projection $X\to B$ admits a Kostant section, so (see Remark~\ref{rm:section}) we can re-write this embedding as
\begin{equation}\label{eq:HitchinFullyFaithful}
    D(\Pic^\tau(X/B))\hookrightarrow D(X).
\end{equation}

\subsubsection{The relation to the Langlands duality for Hitchin systems}
Let $\check G$ be the group Langlands dual to $G$ (thus $\check G$ is of adjoint type and its Dynkin diagram is obtained from that of $G$ by reversing the arrows). It is easy to identify $\Hitch_G$ with $\Hitch_{\check G}$. Consider the `dual' Hitchin system $\check p:T^*\Bun_{\check G}\to\Hitch_G$.

\emph{The Langlands duality for Hitchin systems\/} in its naive form predicts that there is a coherent sheaf $\tL$ on $T^*\Bun_{\check G}\times_{\Hitch_G}T^*\Bun_G$ such that the corresponding integral transform
\begin{equation}\label{eq:HitchinLanglands}
    D(T^*\Bun_{\check G})\to D(T^*\Bun_G):\cF\mapsto Rp_{2,*}(\tL\otimes^L p^*_1\cF)
\end{equation}
is an equivalence of categories. We do not expect this to hold literally but we expect this to hold over a big open subspace of
the Hitchin base $\Hitch_G$. Generically over the base, the conjecture is settled in~\cite[Thm.~A]{DonagiPantev}.

Let $B$ and $\cX$ be as above. Set $\check\cX:=T^*\Bun_{\check G}\times_{\Hitch_G}B$. Since $\check G$ is of adjoint type, the connected components of $\Bun_{\check G}$ are indexed by
\[
    Z':=\pi_1(\check G)=\Hom(Z(G),\gm),
\]
so the same is true for $\check\cX$: $\check\cX=\coprod_{i\in Z'}\check\cX^{(i)}$. Moreover, $\check\cX$ is a scheme projective over $B$.

Restricting the equivalence~\eqref{eq:HitchinLanglands} to $B$, we expect an equivalence of categories
\[
\prod_{i\in Z'}D(\check\cX^{(i)})=D(\check\cX)\to D(\cX)=\prod_{i\in Z'}D^{(i)}(\cX)
\simeq\prod_{i\in Z'}D(X),
\]
where $D^{(i)}(\cX)$ is the full subcategory of $D(\cX)$ of objects on which $Z(G)$ acts according to $i$. Denote by $Y$ the smooth locus of the projection $\check\cX^{(0)}\to B$. Restricting the above equivalence of categories to the unity of $Z'$ and composing it with the natural inclusion $D(Y)\to D(\check\cX^{(0)})$, we get a fully faithful functor
\[
D(Y)\hookrightarrow D(X).
\]
We expect this to coincide with~\eqref{eq:HitchinFullyFaithful}.

\subsection{Compactified Jacobians} In~\cite{ArinkinJacobians}, the following situation is considered. Let $\cM$ be the moduli stack of integral projective curves of fixed genus $g$ with planar singularities.
Denote by $\cC\to\cM$ the universal curve over $\cM$. We then have the universal Jacobian $\cJ:=\Pic^\tau(\cC/\cM)$ and the universal compactified Jacobian $\overline{\cJ}:=\overline{\Pic^\tau}(\cC/\cM)$. Since our results are local in smooth topology over $B$, they immediately extend to the case when $B$ is a smooth stack. Thus
we can take $B=\cM$, $A=\cJ$, and $X=\overline{\cJ}$. Then Theorem~\ref{th:PartialFourier}\eqref{th:PartFourier2} becomes essentially Theorem~10 of~\cite{ArinkinJacobians}, which implies the main results of~\cite{ArinkinJacobians}.

\section{The relative Picard stack}\label{sect:relativePic}
In this section, we summarize the properties of the relative Picard space. The key results go back to Grothendieck's talks~\cite{FGAV,FGAVI}. The theory of relative Picard scheme is developed in detail in~\cite{KleimanFGA}, which we use as our primary reference.

\subsection{} Let $B$ be a locally Noetherian scheme, and let $\pi_X:X\to B$ be a flat projective morphism. In addition, suppose that the geometric fibers of $\pi_X:X\to B$ are non-empty, reduced, and connected.

\begin{lemma}\label{lm:dirimagestructsheaf} The natural morphism $\cO_B\to\pi_{X,*}\cO_X$ is an isomorphism.
\end{lemma}
\begin{proof} See~\cite[Exercise~3.11]{KleimanFGA} (the proof of exercise is given in~\cite[Appendix~A]{KleimanFGA}).
\end{proof}

\begin{corollary}\label{cor:cohflat}
$\cO_B\simeq\pi_{X,*}\cO_X$ holds universally. That is,
for any morphism of locally Noetherian schemes $B'\to B$, the natural morphism $\cO_{B'}\to\pi'_{X,*}\cO_{X'}$ is an isomorphism. Here
$X':=X\times_BB'$ and $\pi'_X:X'\to B'$ is the natural projection. \qed
\end{corollary}

\begin{remark*}
    In particular, $\pi_X$ is cohomologically flat in dimension zero (cf.~\cite[Sect.~8.1]{BLR_NeronModels}).
\end{remark*}

Consider the relative Picard stack $\cPic(X/B)$ of $X$ over $B$. Precisely, for any $B$-scheme $S$, the groupoid of $B$-morphisms $S\to\cPic(X/B)$ is the groupoid of line bundles over $X\times_BS$.  Applying the above lemma and~\cite[Thm~4.6.2.1]{LaumonMoretBailly}, we see that $\cPic(X/B)$ is an algebraic stack locally of finite type over $B$.

Consider the functor associating to a $B$-scheme $S$ the set of isomorphism classes of line bundles over $X\times_BS$. Denote by $\Pic(X/B)$ the fppf sheafification of this functor. By~\cite[Thm.~7.3]{ArtinAlgebraization} (given as Theorem~4.18.6 in~\cite{KleimanFGA}),
$\Pic(X/B)$ is represented by an algebraic space locally of finite type.

Assume that we have a section $\zeta_X:B\to X$ of $\pi_X$. For any $B$-scheme $S$, we can then identify $\Pic(X/B)(S)$ with the
set of isomorphism classes of line bundles over $X\times_BS$ together with a $\zeta_X$-rigidification; see Definition~2.8, Theorem~2.5, and Lemmas~2.9, 2.10 in~\cite{KleimanFGA}. In particular, $\Pic(X/B)$ is a fine moduli space (that is, there is a universal line bundle over $\Pic(X/B)\times_BX$) and we have $\cPic(X/B)\simeq\Pic(X/B)\times {\mathrm B}(\gm)$.

In general, we claim that $\pi_X$ has a section \'etale locally over $B$. Indeed, the subset of $X^{sm}$ where the fiber of $\pi_X$ is smooth is open in $X$ by~\cite[Thm.~12.2.4(iii)]{EGAIV-3}. Also, the restriction of $\pi_X$ to $X^{sm}$ is smooth by~\cite[Thm.~17.5.1]{EGAIV.4}. Next, $\pi_X(X^{sm})=B$ because the fibers of $\pi_X$ are geometrically reduced. Now it follows from~\cite[Cor.~17.16.3(ii)]{EGAIV.4} that $\pi_X$ has a section \'etale locally over $B$. Therefore, $\cPic(X/B)$ is a $\gm$-gerbe over $\Pic(X/B)$.

\subsection{Numerically trivial line bundles}
Note that the formation of $\Pic$ commutes with base changes: if $S\to B$ is a locally Noetherian $B$-scheme, then  \[\Pic(X\times_BS/S)=\Pic(X/B)\times_BS.\] In particular, the fiber of $\Pic(X/B)$ over a geometric point $s\to B$ is equal to $\Pic(X_s):=\Pic(X_s/\spec k(s))$, where $k(s)$ is the function field of $s$. In this case, $\Pic(X_s)$ is known to be a scheme by \cite[Cor.~6.6]{FGAV} (cf. Corollary~4.18.3 in~\cite{KleimanFGA}). Denote by $\Pic^0(X_s)$ the neutral connected component of the group scheme $\Pic(X_s)$. Set
\[
    \Pic^\tau(X_s):=\bigcup_{d>0}\hat d^{-1}(\Pic^0(X_s)),
\]
where $\hat d:\Pic(X_s)\to\Pic(X_s)$ is the multiplication by $d$. The line bundles over~$X_s$ corresponding to points of $\Pic^\tau(X_s)$ are called \emph{numerically trivial line bundles over~$X_s$}.

\begin{definition}\label{def:FibNumTriv}
    A line bundle $\ell$ over $X$ is \emph{fiberwise numerically trivial\/} (with respect to the morphism $\pi_X:X\to B$) if for all geometric points $s\to B$ the restriction of $\ell$ to $X_s$ is numerically trivial.
\end{definition}

Define a subfunctor $\Pic^\tau(X/B)$ of $\Pic(X/B)$ as follows
\begin{multline*}
    \Pic^\tau(X/B)(S)=\\ \{f\in\Pic(X/B)(S): \text{for all geometric points $s\to S$, } f(s)\in\Pic^\tau(X_s)\}.
\end{multline*}
It is easy to see that this definition coincides with one given in~\cite[Exp.~XIII, Sect.~4]{SGA6}. By~\cite[Exp.~XIII, Thm.~4.7(i)]{SGA6}, $\Pic^\tau(X/B)$ is an open subfunctor of $\Pic(X/B)$. By~\cite[Exp.~XIII, Thm.~4.7(iii)]{SGA6}, $\Pic^\tau(X/B)$ is of finite type over $B$ (cf.~also~\cite[Sect.~6]{KleimanFGA}).

Denote by
\[
    \cPic^\tau(X/B):=\cPic(X/B)\times_{\Pic(X/B)}\Pic^\tau(X/B)
\]
the stack classifying fiberwise numerically trivial line bundles; it is a $\gm$-gerbe over $\Pic^\tau(X/B)$. Thus $\cPic^\tau(X/B)$ is of finite type over $B$ as well. Clearly, the formation of $\Pic^\tau(X/B)$ and $\cPic^\tau(X/B)$ commutes with base changes.

\begin{remark}
A line bundle $\ell\to X$ corresponds to a section of $\cPic^\tau(X/B)$ if and only if the following condition is satisfied: \emph{for every geometric fiber of $\pi_X:X\to B$ and for every closed curve $C$ in this geometric fiber, the degree of the pullback of $\ell$ to $C$ is zero.} This follows from~\cite[Thm.~6.3]{KleimanFGA} or~\cite[Exp.~XIII, Thm.~4.6]{SGA6}.
\end{remark}

\subsection{Triviality locus}
Consider the zero section $B\to\Pic(X/B)$ corresponding to the trivial line bundle. Clearly, it factorizes as
\begin{equation}\label{eq:factorizezero}
B\xrightarrow{\iota'}\Pic^\tau(X/B)\to\Pic(X/B).
\end{equation}

\begin{proposition}\label{pp:Cartesian}
\stepzero\noindstep
The zero section $\iota'$ fits into a Cartesian diagram
\begin{equation*}
\begin{CD}
B\times {\mathrm B}(\gm) @>\iota>> \cPic^\tau(X/B)\\
@VVV @VVV\\
B @>\iota'>>\Pic^\tau(X/B).
\end{CD}
\end{equation*}

\noindstep\label{pp:Cart2} A line bundle $\ell$ over $X\times_BS$ is isomorphic to the pullback of a line bundle over $S$ if and only if the corresponding morphism $s_\ell:S\to\Pic(X/B)$ factors through the zero section $B\to\Pic(X/B)$.
\end{proposition}
\begin{proof}
The top arrow $\iota$ is defined as follows. Note that a $B$-morphism from a~$B$-scheme $S$ to $B\times {\mathrm B}(\gm)$ is just a line bundle over $S$. The morphism $\iota$ sends such a line bundle $\ell$ to the line bundle $p_2^*\ell$ over $X\times_BS$. The fact that the diagram is Cartesian can be checked locally in \'etale topology over $B$, so we may assume that $\pi_X:X\to B$ has a section, in which case the statement is easy. The second part of the proposition follows easily from the first part.
\end{proof}

Let us prove the following claim:
\begin{proposition}\label{pr:zerosection}
The morphism $\iota':B\to\Pic^\tau(X/B)$ is a closed embedding.
\end{proposition}

\begin{remark} Thus,~\eqref{eq:factorizezero} is a decomposition of the zero section as a closed embedding followed by an open embedding.
\end{remark}

\begin{corollary}\label{cor:pictausep}
    The projection $\Pic^\tau(X/B)\to B$ is a separated morphism of algebraic spaces.
\end{corollary}
\begin{proof}
The argument is well known, at least in case of group schemes: the diagonal morphism is a base change of the zero section.
\end{proof}

We start the proof of Proposition~\ref{pr:zerosection} with the following simple observation.
\begin{lemma}\label{lm:numtrivialtrivial}
    Let $Z$ be a connected projective reduced scheme over an algebraically closed field $k$, let $\ell$ be a numerically trivial line bundle over $Z$. If $H^0(Z,\ell)\ne0$, then $\ell$ is a trivial line bundle.
\end{lemma}
\begin{proof}
    Let $s$ be a non-zero section of $\ell$, let $Z'$ be an irreducible component of $Z$. If $s|_{Z'}\ne0$, then $s|_{Z'}$ is an injection $\cO_{Z'}\to\ell|_{Z'}$ because $Z'$ is reduced. Since for some $l>0$ the line bundles $\cO_{Z'}$ and $(\ell|_{Z'})^{\otimes d}$ are algebraically equivalent, they have the same Hilbert polynomial, and we see that $(s|_{Z'})^{\otimes d}$ is an isomorphism. It follows that $s|_{Z'}$ is an isomorphism.

    Thus $\{z\in Z:s_z\ne0\}$ is a union of components of $Z$. We see that this set is closed. Since it is also open, we see that it coincides with $Z$, so that $s$ is a nowhere vanishing section of $\ell$.
\end{proof}

Essentially, Lemma~\ref{lm:numtrivialtrivial} implies Proposition~\ref{pr:zerosection} by semicontinuity. Let us analyze this
situation carefully (we will need this analysis in the future).

Let $\ell$ be a fiberwise numerically trivial line bundle over $\pi_X:X\to B$. Consider the derived pushforward $R\pi_{X,*}\ell$, and
put
\[
    \cF:=\HOM(R\pi_{X,*}\ell,\cO_B)=H^0(R\HOM(R\pi_{X,*}\ell,\cO_B)).
\]
Clearly, $\cF$ is a coherent sheaf on $B$. Note that its formation commutes with (not necessarily flat) base change because $R\HOM(R\pi_{X,*}\ell,\cO_B)$ is a perfect complex supported in non-positive degrees (see for example the second theorem and Lemma~1 in~\cite[Sect.~5]{MumfordAbelian}).

\begin{proposition}[{cf.~\cite[Exercise~4.2]{KleimanFGA}}]\label{pr:trivlocus}
Consider the scheme theoretic support $\supp(\cF)\subset B$; it is a closed subscheme of $B$.\\
\stepzero
\noindstep\label{pr:trivlocus2} $\cF$ is isomorphic to the pushforward of a line bundle over $\supp(\cF)$;\\
\noindstep\label{pr:trivlocus1} For a morphism $\phi:Z\to B$, the pullback of $\ell$ to $Z\times_BX$ is isomorphic to the pullback of a line bundle over $Z$ if and only if $\phi$ factors through the closed embedding $\supp(\cF)\hookrightarrow B$.
\end{proposition}
\begin{proof}
We need the following general statement.
\begin{lemma}
    A coherent sheaf $\cF$ on a locally Noetherian scheme $B$ is isomorphic to a push-forward of a line bundle over $\supp(\cF)$ if and only if the fiber of $\cF$ at any point $b\in B$ is at most one-dimensional.
\end{lemma}
\begin{proof}
    Let $\cF$ be such that the fiber of $\cF$ at any point $b\in B$ is at most one-dimensional. Since every coherent sheaf is isomorphic to the pushforward of a sheaf on its support, we may assume that $\supp(\cF)=B$, so the fiber of $\cF$ at any point $b\in B$ is one-dimensional. Let $s$ be a section of $\cF$ in a neighbourhood of $b\in B$ such that $s$ does not vanish at $b$. By Nakayama's Lemma, shrinking the neighbourhood of $b$, we may assume that $s$ generates $\cF$. Thus $\cF$ is isomorphic in a neighbourhood of $b$ to the quotient of $\cO_B$ by a sheaf of ideals. Since we assumed that $\supp(\cF)=B$, this sheaf of ideals should be zero, and we see that $\cF$ is a line bundle. The converse statement is trivial.
\end{proof}

Let us return to the proof of the proposition. By the previous lemma, part~\eqref{pr:trivlocus2} is equivalent to the claim that the fiber of $\cF$ at any point $b\in B$ is at most one-dimensional.
Since the construction of $\cF$ commutes with the base change, it amounts to checking that
\[
    \dim H^0(X_b,\ell_b)\le 1;
\]
here $X_b:=X\times_Bb$ is the fiber of $\pi_X$ and $\ell_b$ is the pullback of $\ell$ to $X_b$. This claim follows from
Lemma~\ref{lm:numtrivialtrivial} and Corollary~\ref{cor:cohflat}.

Let us now prove part~\eqref{pr:trivlocus1}. By base change, we may assume without losing generality that $Z=B$ and $\phi=\Id_B$.
Thus, we need to show that $\ell$ is isomorphic to the pullback of a line bundle over $B$ if and only if $\supp(\cF)=B$.

Clearly, $\ell$ is isomorphic to the pullback of a line bundle over $B$ if and only if there exists a line bundle $\ell_B$ over $B$ and a morphism
\[
    f:\pi_X^*\ell_B\to\ell
\]
that is non-zero on fibers over any point of $X$. By Lemma~\ref{lm:numtrivialtrivial}, this is equivalent to the condition that the
morphism is not identically zero on the fiber of $\pi_X$ over each point $b\in B$. We have
\begin{multline*}
    \Hom(\pi_X^*\ell_B,\ell)\simeq\Hom_{D(B)}(\ell_B,R\pi_{X,*}\ell)\simeq\\
    \Hom_{D(B)}(R\HOM(R\pi_{X,*}\ell,\cO_B),\ell_B^\vee)\simeq
    \Hom(\cF,\ell_B^\vee),
\end{multline*}
where the last identification follows from the fact that $R\HOM(R\pi_{X,*}\ell,\cO_B)$ is supported in non-positive degrees. Note also, that the above identifications are compatible with pullbacks, so $f$ is not identically zero on the fiber over any point $b\in B$ if and only if the corresponding morphism $h:\cF\to\ell_B^\vee$ is non-zero at every $b\in B$. Since we know that $\cF$ is a line bundle over $\supp\cF$, the existence of such $h$ and $\ell_B$ is equivalent to $\supp\cF=B$. This implies the statement.
\end{proof}

The following useful corollary is known as the see-saw principle.
\begin{corollary}\label{cor:see-saw}
Let $B$ be reduced. A line bundle over $X$ is isomorphic to the pullback of a line bundle over $B$ if and only if its restrictions to the geometric fibers of $\pi_X$ are trivial.
\end{corollary}
\begin{proof}
Let $\ell$ be a line bundle over $X$ whose restrictions to the geometric fibers of $\pi_X$ are trivial. Let $\cF$ be as in the proposition. It follows from part~\eqref{pr:trivlocus1} of the proposition that the support of $\cF$ coincides with $B$ set-theoretically. Since $B$ is reduced, this support coincides with $B$ scheme-theoretically. Using part~\eqref{pr:trivlocus1} of the proposition again, we see that $\ell$ is isomorphic to a pullback from $B$.

The converse statement is obvious.
\end{proof}

\begin{proof}[Proof of Proposition~\ref{pr:zerosection}]
By definition, it suffices to check that for any scheme $S$ and any morphism $s:S\to\Pic^\tau(X/B)$, the fiber product
\[
    S\times_{\Pic^\tau(X/B)}B\to S
\]
is a closed embedding. Since $\Pic^\tau(X/B)$ is of finite type over the locally Noetherian scheme $B$, we can restrict
ourselves to locally Noetherian schemes $S$.
Recall that $\Pic^\tau(X\times_BS/S)=\Pic^\tau(X/B)\times_BS$. Thus replacing $B$ with $S$ (and $X$ with $X\times_BS$), we may assume that $S=B$ and the morphism $s:B\to\Pic^\tau(X/B)$ is a section. Since our statement is local in \'etale topology on $B$, we may assume that $\pi_X:X\to B$ has a section, so $\Pic^\tau(X/B)$ is a fine moduli space. Then the section~$s$ comes from a fiberwise numerically trivial line bundle $\ell$ over $\pi_X:X\to B$. Let $\cF$ be as in Proposition~\ref{pr:trivlocus}. This proposition together with Proposition~\ref{pp:Cartesian}\eqref{pp:Cart2} gives a Cartesian diagram
\[
\begin{CD}
\supp\cF @>>> B\\
@VVV @V s VV\\
B @>\iota'>>\Pic^\tau(X/B)
\end{CD}
\]
and the claim follows.
\end{proof}

Recall that $\cL$ is the universal line bundle over $\cPic^\tau(X/B)\times_BX$; put
\[
    \cF_{univ}:=\HOM(Rp_{1,*}\cL,\cO_{\cPic^\tau(X/B)}).
\]

Proposition~\ref{pr:trivlocus} is local over the base, so we can apply it to $\cL$ and the projection $\cPic^\tau(X/B)\times_BX\to\cPic^\tau(X/B)$ obtaining the following
\begin{corollary}\label{cor:Funiv}
The sheaf $\cF_{univ}$ is a direct image of a line bundle over $B\times {\mathrm B}(\gm)$ under $\iota$.
\end{corollary}
\begin{remark}\label{rm:Funiv}
It is easy to see that Proposition~\ref{pr:trivlocus} is equivalent to the above corollary. In fact, the line bundle in the corollary is easy to describe explicitly: it is the trivial line bundle $\cO_B$ over $B$ on which $\gm$ acts with weight $-1$.
(We identify quasicoherent sheaves on $B\times {\mathrm B}(\gm)$ with quasicoherent sheaves
on $B$ equipped with a linear action of $\gm$.)
\end{remark}

\section{Quasi-constructible sets and quasi-subgroups}\label{sect:quasi-constr}
\begin{conventions*}
All schemes are assumed to be locally Noetherian, and all morphisms of schemes are locally of finite type. When we work over a fixed field
$k$, all schemes are assumed to be locally of finite type over $k$; for two such schemes $X$ and $Y$,
we write simply $X\times Y$ for $X\times_k Y$.

Let $X$ be a scheme. Given a schematic point $x$ of $X$, we denote by $k(x)$ its residue field, and use $x$ as a shorthand for
$\spec k(x)$.
\end{conventions*}

Let $p:Y\to X$ be a morphism of finite type, and let $\ell$ be a line bundle over $Y$. We view
$Y$ as a family of schemes over $X$, and $\ell$ as a family of line bundles over these schemes.
(In our applications, $p$ is faithfully flat but we do not need the assumption at this point.)
Consider the \emph{triviality locus} of this family:
\begin{equation}\label{eq:triviality locus}
K':=\{x\in X:\ell|_{x\times_XY}\simeq\cO_{x\times_XY}\}\subset X.
\end{equation}
In particular, $K'\supset X-p(Y)$.

If $p$ is a flat projective morphism with reduced and connected geometric fibers, then $K'\subset X$ is a locally closed subset (combine openness of $\cPic^\tau(Y/X)$ in $\cPic(Y/X)$ with Proposition~\ref{pr:trivlocus}). In general, $K'$ is not locally closed, or even constructible. However, one can show that $K'$ is always a (possibly infinite) union of locally closed subsets of $X$; see Proposition~\ref{pp:linebtriv} below. We call such subsets \emph{quasi-constructible}.

We prefer to work with the \emph{finite order locus}
\begin{equation}\label{eq:finite order locus}
K:=\{x\in X: \exists\; d>0 \text{ such that }\ell^{\otimes d}|_{x\times_XY}\simeq\cO_{x\times_XY}\}
\end{equation}
instead of $K'$. The advantage of $K$ is that it behaves better under base changes (Corollary~\ref{co:finite order base change}).

We then apply this framework in the following situation. Let $A$ be a group scheme over a field $k$, and let $\Theta$ be a line bundle over $A$. We can then consider the shifts of $\Theta$ (by the action of $A$ on itself by translations); the shifts form an $A$-family of line bundles over $A$ (thus $X=A$ and $Y=A\times A$). The triviality locus $K$ of this family is preserved by the group
operation on $A$; we use the term `quasi-subgroup' in such situation. We give a precise definition of quasi-subgroups below and provide a classification (assuming $k$ is algebraically closed) in Proposition~\ref{pp:quasigroups}. The main result of this section is Corollary~\ref{cor:ampleshiftsaffine}.

\subsection{Quasi-constructible sets}
We denote by $X^{top}$ the underlying topological space of a scheme $X$.

\begin{definition}
A subset $K\subset X^{top}$ is \emph{quasi-constructible} if it is a possibly infinite union of locally closed subsets.
\end{definition}

Obviously, a subset $K\subset X^{top}$ is quasi-constructible if and only if for every $x\in K$ there is a non-empty open subscheme $Z\subset\overline{\{x\}}$ such that $Z^{top}\subset K$. Note the following properties of quasi-constructible sets.

\begin{proposition}\label{pp:qc}
\begin{enumerate}
\item\label{it:pp:qc1} The class of quasi-constructible sets is invariant under finite intersections and arbitrary unions.

\item\label{it:pp:qc2} Let $f:Y\to X$ be a morphism of schemes. If $K\subset X^{top}$ is quasi-\linebreak constructible, then so is $f^{-1}(K)\subset Y^{top}$.

\item\label{it:pp:qc3} Let $f:Y\to X$ be a morphism of schemes. If $K\subset Y^{top}$ is quasi-constructible, then so is $f(K)\subset X^{top}$.

\item\label{it:pp:qc:product} Suppose we are given a Cartesian diagram
\[\xymatrix{Y_1\times_X Y_2\ar[r]\ar[d]&Y_2\ar[d]\\ Y_1\ar[r]& X}\]
of schemes and quasi-constructible subsets $K_i\subset Y_i^{top}$, $i=1,2$. Let
\[
    (K_1\times_X K_2)^{qc}\subset(Y_1\times_X Y_2)^{top}
\]
be the preimage of $K_1\times K_2$ under the natural continuous map
\[
    (Y_1\times_X Y_2)^{top}
    \to Y_1^{top}\times Y_2^{top}.
\]
Then $(K_1\times_X K_2)^{qc}$ is quasi-constructible.
\end{enumerate}
\end{proposition}
\begin{proof}
Parts~\eqref{it:pp:qc1},~\eqref{it:pp:qc2}, and~\eqref{it:pp:qc:product} are straightforward. Part~\eqref{it:pp:qc3} follows
from Chevalley's Theorem: the image of a constructible set under a morphism of finite type is constructible, see~\cite[Thm.~IV.1.8.4]{EGAIV-1}. (Recall that $f$ is assumed to be locally of finite type.)
\end{proof}

\begin{definition} Let $X$ be a scheme, and let $K\subset X^{top}$ be a quasi-constructible subset. A subset $S\subset K$ is a \emph{component} if it is maximal among subsets of $K$ that are irreducible and locally closed in $X^{top}$.
\end{definition}

\begin{proposition} Every quasi-constructible set $K$ is a (possibly infinite) union of its components.
\end{proposition}
\begin{proof} It is enough to prove that the union of a chain of irreducible and locally closed subsets of $X$ is also irreducible and locally closed.

So, let $S_\alpha$ be such a chain. Let $\overline S_\alpha$ be the Zariski closure of $S_\alpha$, and let $x\in\cup_\alpha\overline S_\alpha$ be a point. The decreasing sequence of prime ideals corresponding to $\overline S_\alpha$ in the local ring $\cO_{X,x}$ stabilizes because this local ring is finite-dimensional (being Noetherian). It follows from irreducibility of schemes $\overline S_\alpha$ that the chain $\overline S_\alpha$ stabilizes as well. The rest of the proof is easy and is left to the reader.
\end{proof}

The following lemma will be used in the proof of Proposition~\ref{pp:quasigroups}.

\begin{lemma}\label{lm:q-constr}
  Let $K\subset X^{top}$ be quasi-constructible. Assume that the dimensions of components of $K$ are bounded. Let $H\subset K$ be a component of maximal dimension, and let $H'\subset X^{top}$ be constructible set with irreducible closure such that $H\subset H'\subset K$. Then $\overline H=\overline{H'}$.
\end{lemma}
\begin{proof}
    Since $H'$ is constructible with irreducible closure, there is a unique point $\xi'\in H'$ such that $\bar\xi'=\overline{H'}$. If $\xi$ is the generic point of $H$, then $\dim\bar\xi\ge\dim\overline{\xi'}$, thus $\xi=\xi'$.
\end{proof}

\subsection{Triviality locus}\label{sc:triviality}
Let $X$ be a (locally Noetherian) scheme, and let $p:Y\to X$ be a morphism of finite type. Let $\ell$ be a line bundle over
$Y$.

\begin{proposition}\label{pp:linebtriv}
The triviality locus $K'\subset X^{top}$ of $\ell$, given by \eqref{eq:triviality locus}, is quasi-constructible.
\end{proposition}
\begin{proof} Follows from Lemma~\ref{lm:linebgentriv}.
\end{proof}

\begin{lemma}\label{lm:linebgentriv}
In the situation of Proposition~\ref{pp:linebtriv}, assume that $X$ is integral. Let $\xi=\spec k(X)$ be the
generic point of $X$; here $k(X)$ is the field of rational functions on $X$. Suppose the
pullback of $\ell$ to $\xi\times_X Y$ is trivial. Then there is an open subscheme $X'\subset X$ such that
\[\ell|_{p^{-1}(X')}\simeq\cO_{p^{-1}(X')}.\]
\end{lemma}
\begin{proof}
Without loss of generality, we may assume that the scheme $X$ is affine. Let $s\in H^0(\xi\times_X Y,\ell)$ be a trivialization, and let
$Y=\cup_i Y_i$ be a finite cover of $Y$ by affine subschemes. For each
$i$, there exists a non-empty open subset $U_i\subset X$ and a~trivialization
\[
    s_i:\ell|_{U_i\times_X Y_i}\xrightarrow{\simeq}\cO_{U_i\times_X Y_i}
\]
such that $s$ and $s_i$ agree on $\xi\times_X Y_i$. For every pair $i,j$, there exists a non-empty open subset
\[U_{ij}\subset U_i\cap U_j\]
such that \[s_i|_{U_{ij}\times_X (Y_i\cap Y_j)}=s_j|_{U_{ij}\times_X (Y_i\cap Y_j)}.\]
Put $X'=\bigcap U_{ij}$. The trivializations $s_i$ glue to a trivialization of $\ell|_{X'\times_X Y}$.
\end{proof}

\begin{corollary}\label{co:finite order} The finite order locus $K\subset X^{top}$ of $\ell$, given by~\eqref{eq:finite order locus},
is quasi-constructible.
\end{corollary}
\begin{proof}
Follows from Proposition~\ref{pp:linebtriv} and Proposition~\ref{pp:qc}\eqref{it:pp:qc1}.
\end{proof}

Note the following lemma.

\begin{lemma} \label{lm:finite order extension} Suppose $k$ is a field, $Y$ is a $k$-scheme of finite type, and $\ell$ is a
line bundle over $Y$. Let $k'\supset k$ be a field extension. Denote the pullback of $\ell$ to
\[Y_{k'}:=\spec k'\times_{\spec k} Y\] by $\ell_{k'}$. Then $\ell$ has finite order (that is, $\ell^{\otimes d}\simeq\cO_Y$ for some $d>0$) if and only if $\ell_{k'}$ has finite order.
\end{lemma}
\begin{proof} The `only if' direction is obvious. Let us prove the `if' direction. Without loss of generality, we may assume that $\ell_{k'}$ itself is trivial. Also note that
the isomorphism $\ell_{k'}\simeq\cO_{Y_{k'}}$ is defined over a finitely generated extension of $k$, so we may assume that $k'$ is finitely generated over $k$.

Let us choose an integral scheme $X$ of finite type over $k$ with the property that $k(X)\simeq k'$. The line bundle $\cO_X\boxtimes\ell$
over $X\times Y=X\times_k Y$ satisfies the hypotheses of Lemma~\ref{lm:linebgentriv} with respect to the projection
$p_1:X\times Y\to X$. By this
lemma, we may assume, after replacing $X$ by an open subscheme, that $\cO_X\boxtimes\ell\simeq\cO_{X\times Y}$.
Let $x\in X$ be a closed point with residue field $k(x)$; since the pullback of $\ell$ to
$x\times Y$ is trivial, we may replace $k'$ with $k(x)$ and assume that $k'\supset k$ is a finite extension.

Finally, let $p_2:Y_{k'}\to Y$ be the projection. Put $d=[k':k]$, and note that
\[\ell^{\otimes d}\simeq\det(p_{2,*}p_2^*\ell)\simeq\det(p_{2,*}\cO_{Y_{k'}})\simeq\det(k'\otimes_k\cO_Y)\simeq\cO_Y.\]
\end{proof}

\begin{corollary}\label{co:finite order base change}
Let $K$ be as in Corollary~\ref{co:finite order}. Let $X_1$ be a scheme and $f:X_1\to X$ be a morphism. Put $Y_1:=X_1\times_X Y$,
and let the line bundle $\ell_1$ over $Y_1$ be the pullback of $\ell$. Then
\[\{x_1\in X_1: \exists\; d>0 \text{ such that }\ell_1^{\otimes d}|_{x_1\times_{X_1}Y_1}\simeq\cO_{x_1\times_{X_1}Y_1}\}=f^{-1}(K).\]
That is, the finite order locus $K$ behaves well under base changes.
\qed
\end{corollary}

\subsection{Quasi-subgroups}\label{sect:quasi-subgr}
Suppose now that $A$ is a group scheme of finite type over a field $k$. Let $\mu_A:A\times A\to A$ be the group law,
let $1\in A(k)$ be the identity, and let $\iota_A:A\to A$ be the inversion. Denote by $\nu_A:A\times A\to A$ the
``division'' morphism, that is, $\nu_A:=\mu_A\circ(\Id_A\times\iota_A)$.

\begin{definition} A quasi-constructible set $K\subset A^{top}$ is a \emph{quasi-subgroup} if $1\in K$, $K$ is invariant under $\iota_A$, and the image of $(K\times K)^{qc}\subset(A\times A)^{top}$ under $\mu_A$ is contained in $K$. (Here $(K\times K)^{qc}=(K\times_{\spec k}K)^{qc}$ is as in Proposition~\ref{pp:qc}\eqref{it:pp:qc:product}.) Equivalently, the last condition means that for all $z\in A\times A$, if $p_1(z)\in K$ and $p_2(z)\in K$, then $\mu_A(z)\in K$.
\end{definition}

\begin{proposition}\label{pp:quasigroups}
Suppose $k$ is algebraically closed. Let $K\subset A^{top}$ be a quasi-subgroup.  Then there exists a closed connected
reduced subgroup scheme $H\subset A$ such that $K$ is a (possibly infinite) union of cosets of $H$. More precisely, $K\subset N(H)^{top}$, where $N(H)$ is
the normalizer of $H$, and $K$ is the preimage of a (possibly infinite) subgroup in $(N(H)/H)(k)\subset(N(H)/H)^{top}$.
\end{proposition}

\begin{proof}
We may assume that $A$ is reduced. Let $H\subset K$ be a component of $K$ of maximal dimension. Without loss of generality,
we may assume that $1\in H$; otherwise, we replace $H$ by $Hh^{-1}$ for any $h\in H(k)$. In what follows, we do
not distinguish between $H$ and the corresponding reduced locally closed subscheme of~$A$.

Note that $H\times H\subset A\times A$ is irreducible because $k$ is algebraically closed.
Let $H\cdot H^{-1}$ be the image of $H\times H$ under the map $\nu_A$. Then $H\cdot H^{-1}\subset K$ is constructible, and its closure is irreducible. Since $H\cdot H^{-1}\supset H$, Lemma~\ref{lm:q-constr} implies that $H\cdot H^{-1}\subset\overline{H}$. Therefore, $\overline{H}\subset A$ is a~subgroup.

For any $h\in H(k)$, both $H$ and $Hh^{-1}$ are open subsets of $\overline{H}$. Therefore, their union $H\cup Hh^{-1}$ is open in $\overline{H}$ and hence locally closed in $A$. Since this union is contained in $K$, maximality of $H$ implies that $Hh^{-1}\subset H$. Therefore, $H\subset\overline{H}$ is a~subgroup. Since it is a dense open subgroup of $\overline{H}$, we see that
\[
    \overline H=H.
\]

Finally, let $K_1\subset K$ be another component. Pick $x\in K_1(k)$, and note that the subset $H\cdot K_1\cdot x^{-1}\subset K$ is a constructible subset with irreducible closure. Since the subset contains $H$, by Lemma~\ref{lm:q-constr} it coincides with $H$. Hence $K_1\subset H\cdot K_1=Hx$. Since $K_1\subset K$
is a component, $K_1=Hx$. Similarly, $K_1=xH$, so that $x\in N(H)$. The claim follows.
\end{proof}

Clearly, the subgroup scheme $H$ in the statement of Proposition~\ref{pp:quasigroups} is uniquely determined by $K$. We call it
\emph{the neutral component} of $K$.

\subsection{Shifts of line bundles}\label{sect:shiftsoflb}
As before, $A$ is a group scheme of finite type over a~field $k$. Fix a line bundle $\Theta$ over $A$.

Given a schematic point $x\in A$, put
\[A_{k(x)}:=x\times A=\spec k(x)\times A.\] Define the shift $\mu_x:A_{k(x)}\to A$ as the composition of the canonical morphism $A_{k(x)}\to A\times A$ and the group law. Let $p_x:A_{k(x)}\to A$ be the projection.
Put
\begin{equation}\label{eq:shifts}
    K_\Theta:=\{x\in A^{top}:\exists\; d>0\text{ such that }\mu_x^*\Theta^{\otimes d}\simeq p_x^*\Theta^{\otimes d}\}.
\end{equation}
By Corollary~\ref{co:finite order} applied to the first projection $p_1:A\times A\to A$ and the line bundle $\ell:=\mu_A^*\Theta\otimes p_2^*\Theta^{-1}$ over $A\times A$, we know that $K_\Theta$ is quasi-constructible.

\begin{proposition}\label{pr:quasi-subgroup}
$K_\Theta$ is a quasi-subgroup.
\end{proposition}
\begin{proof}
It is obvious that $1\in K_\Theta$.

Let $z$ be a schematic point of $A\times A$ whose projection onto each factor lies in $K_\Theta$. We need to verify that its image under $\mu_A$ belongs to $K_\Theta$ as well. By Corollary~\ref{co:finite order base change}, we can extend the ground field from $k$ to $k(z)$. Thus, we may assume that $z$ is defined over the ground field $k$, in which case the statement is obvious.

Now it is also obvious that $K_\Theta$ is invariant under $\iota_A$.
\end{proof}

\begin{proposition}\label{pr:affine}
Assume that $\Theta$ is an ample line bundle over $A$. If $K_\Theta=A^{top}$, then $A$ is affine.
\end{proposition}
\begin{proof}
Clearly, we may assume that $k$ is algebraically closed. We may also assume that $A$ is reduced (that is, smooth) and connected.
By Lemma~\ref{lm:qaffine} below, it suffices to show that $A$ is quasi-affine. Let $\xi\in A^{top}$ be the generic point.
Replacing $\Theta$ by $\Theta^{\otimes d}$ for a suitable $d$, we may assume that
\[
    \mu_\xi^*\Theta\simeq p_\xi^*\Theta.
\]
Applying Lemma~\ref{lm:linebgentriv} to $p_1:A\times A\to A$ and the line bundle
$\ell:=\mu_A^*\Theta\otimes p_2^*\Theta^{-1}$ over $A\times A$, we see that there exists a dense open subscheme $Z\subset A$ such that $\ell$ is a trivial line bundle over $Z\times A$. We claim that $\ell$ is of the form $p_1^*\ell'$, where $\ell'$ is a~line bundle over $A$. Indeed, let $D_1$, \ldots, $D_n$ be all the codimension one irreducible components of $A-Z$. Since $k$ is algebraically closed, $D_1\times A$, \ldots, $D_n\times A$ are all the codimension one irreducible components of $(A\times A)-(Z\times A)$. Since $A$ is smooth, the trivialization of $\ell$ over $Z\times A$ gives rise to a Weil divisor, and we see that this divisor is a combination of $D_i\times A$. The claim follows.

Restricting $\ell=\mu_A^*\Theta\otimes p_2^*\Theta^{-1}$ to $A\times 1$, we see that $\ell'\simeq\Theta$. Therefore
\[
    \mu_A^*\Theta\simeq p_1^*\Theta\otimes p_2^*\Theta.
\]
Let $\Delta_-:A\to A\times A$ be the anti-diagonal given by $\Delta_-:=(\Id_A\times\iota_A)\circ\Delta_A$, where $\Delta_A$ is the diagonal. Pulling back the above isomorphism via $\Delta_-$, we get
\[
   \cO_A\simeq\Theta\otimes\iota_A^*\Theta.
\]
The right hand side is an ample line bundle; therefore, $\cO_A$ is ample, and we see that $A$ is quasi-affine, as claimed.
\end{proof}

\begin{lemma}\label{lm:qaffine}
A quasi-affine algebraic group $A$ of finite type over $k$ is linear.
\end{lemma}
\begin{proof}
Identical to the proof in the affine case. In more detail, the regular representation of $A$ is faithful. As this representation is a union of finite dimensional subrepresentations, we can find a faithful finite dimensional representation of $A$ (cf.~the proof of~\cite[Prop.~1.10]{BorelLinAlgGrps}).
\end{proof}

\begin{corollary}\label{cor:fiberestimate}\label{cor:ampleshiftsaffine}
Suppose that the ground field $k$ is algebraically closed and that $\Theta$ is an ample line bundle over $A$. Let $K_\Theta$ be as in~\eqref{eq:shifts};
denote by $K^0_\Theta\subset K_\Theta$ the neutral component of $K_\Theta$ (see Proposition~\ref{pp:quasigroups}). Then the group scheme $K^0_\Theta$ is affine.
\end{corollary}
\begin{proof}
Apply Proposition~\ref{pr:affine} to the restriction of the line bundle $\Theta$ to $K^0_\Theta$.
\end{proof}

\section{Squarish and cubish line bundles}
\begin{conventions*}
In this section, we work over a fixed field $k$. (The relative version of these results is discussed in the next section.)
Let $X$ be a geometrically reduced and geometrically connected projective scheme over $k$. Let $A$ be a group scheme of finite type over $k$. Assume that we are given an action $\mu:A\times X\to X$.
\end{conventions*}

If $A=X$ is an abelian variety acting on itself by translations, then it follows from the Theorem of the Cube that every line bundle over $X$ satisfies a certain property. For general $A$ and $X$, this property is not automatic and we formalize it in Definition~\ref{def:cubish}, calling the corresponding line bundles \emph{cubish}.

Numerically trivial line bundles over abelian varieties satisfy a stronger property; we call such line bundles \emph{squarish} in general situation (Definition~\ref{def:squarish}). In this section, we study squarish and cubish line bundles.
The main result of this section is Proposition~\ref{pr:mainfiberwise}.

\subsection{Squarish line bundles}
\begin{definition} \label{def:squarish}
A line bundle $\ell$ over $X$ is said to be \emph{squarish} if the line bundle $\mu^*\ell\otimes p_2^*\ell^{-1}$ over $A\times X$
is isomorphic to the pullback of a line bundle over $A$.
\end{definition}

In other words, $\ell$ is squarish if and only if there exists a line bundle $\ell_A$ over $A$ such that
$\mu^*\ell\otimes p_2^*\ell^{-1}\simeq p_1^*\ell_A$.
From Lemma~\ref{lm:dirimagestructsheaf}, we obtain an isomorphism
\[\ell_A\simeq p_{1,*}(\mu^*\ell\otimes p_2^*\ell^{-1});\]
thus, $\ell_A$ is determined by $\ell$.

\begin{proposition}\label{pp:abelian variety}
Assume that $k$ is algebraically closed and suppose $A=X$ is an abelian variety, acting on itself by translations. Then a line bundle $\ell$ is squarish if and only if it is numerically trivial.
\end{proposition}
\begin{proof}
By the first definition in~\cite[Sect.~8]{MumfordAbelian} and the theorem in~\cite[Sect.~13]{MumfordAbelian}, a line bundle is squarish if and only if it corresponds to a point of the dual abelian variety $\check X=\Pic^0(X)$. It remains to show that $\Pic^0(X)=\Pic^\tau(X)$. Let $\ell$ be a line bundle on $X$ such that for some $d>0$ we have $[\ell^{\otimes d}]\in\Pic^0(X)(k)$. By Proposition~2 of~\cite[Sect.~2]{MumfordAbelian}, the group $\Pic^0(X)(k)$ is divisible, so there is $\ell_0\in\Pic^0(X)(k)$ such that $\ell^{\otimes d}\simeq \ell_0^{\otimes d}$. Thus $\ell\otimes\ell_0^{-1}$ has finite order in the group $\Pic^0(X)(k)$, so applying Observation~(v) of~\cite[Sect.~8]{MumfordAbelian}, we see that $[\ell\otimes\ell_0^{-1}]\in\Pic^0(X)(k)$. The statement follows.
\end{proof}

\begin{example}\label{ex:compactified Jacobian}
Let $C$ be an integral projective curve. Let $\Pic^0(C)$ be the Jacobian of~$C$.
If $C$ is singular, $\Pic^0(C)$ is not compact, but it admits a natural compactification: the compactified Jacobian $\overline{\Pic}^0(C)$, introduced by Altman, Iarrobino, and Kleiman in~\cite{AltmanIarrobinoKleiman}.
Points in $\overline{\Pic}^0(C)$ parameterize torsion-free coherent sheaves $\cF$ on $C$ such that the generic rank of
$\cF$ is $1$ and $\chi(\cF)=\chi(\cO_C)$.
The action of $\Pic^0(C)$ on itself by translations extends to an action of $\Pic^0(C)$ on $\overline{\Pic}^0(C)$ given by the tensor product
\[
    (L,\cF)\mapsto L\otimes\cF,\qquad(L\in\Pic^0(C),\cF\in\overline{\Pic}^0(C)).
\]

Now fix $L_0\in\Pic^0(C)$. It defines a line bundle $\ell$ over $\overline{\Pic}^0(C)$ whose fiber over a~torsion-free sheaf $\cF\in\overline{\Pic}^0(C)$ is given by
\[
    \ell_{\cF}:=\detRG(L_0\otimes\cF)\otimes\detRG(\cF)^{-1}.
\]
It is not hard to see that this line bundle is squarish. This claim appears implicitly in~\cite[Prop.~1]{ArinkinJacobians}. If $C$ is smooth, this is a special case of Proposition~\ref{pp:abelian variety}.
\end{example}

The notion of a squarish line bundle can be reformulated as follows. A line bundle $\ell$ over $X$ gives rise to a point
$[\ell]\in\Pic(X)(k)$. The action of $A$ on $X$ induces an action
\[\mu_{\Pic}:A\times\Pic(X)\to\Pic(X)\]
of $A$ on $\Pic(X)$.

\begin{proposition}\label{pr:SquarishFixed}
A line bundle $\ell$ is squarish if and only if $[\ell]\in\Pic(X)(k)$ is fixed by $A$.
\end{proposition}
\begin{proof}
Consider the line bundle $\mu^*\ell\otimes p_2^*\ell^{-1}$ over $X\times A$. As an $A$-family of line bundles over $X$, it gives a morphism $s'_\ell:A\to\Pic(X)$. Now both conditions are equivalent to $s'_\ell$ being the zero morphism (use Proposition~\ref{pp:Cartesian}\eqref{pp:Cart2}).
\end{proof}

Let $a\in A$ be a point. Similarly to Section~\ref{sect:shiftsoflb}, we denote by $\mu_a:X_{k(a)}\to X$ and $p_a:X_{k(a)}\to X$
the shift and the projection, respectively.

\begin{proposition}\label{pr:squarish reduced}
Suppose $A$ is reduced. Then a line bundle $\ell$ over $X$ is squarish if and only if for every point $a\in A$,
we have $\mu_a^*\ell\simeq p_a^*\ell$.
\end{proposition}
\begin{proof}
The `if' direction follows from Corollary~\ref{cor:see-saw} applied to $\mu^*\ell\otimes p_2^*\ell^{-1}$. The `only if' direction is obvious.
\end{proof}

A line bundle $\ell$ is squarish if and only if it is `projectively equivariant' under the action of $A$, in the following
sense:

\begin{proposition}[{cf.~\cite[Sect.~10.2]{PolishchukAbelian}}]\label{pp:squarishequiv} A line bundle $\ell$ is squarish if and only if
there exists a central extension of group schemes
\begin{equation}\label{eq:extension}
    1\to\gm\to\tA\to A\to1
\end{equation}
and a structure of an $\tA$-equivariant line bundle on $\ell$ (here $\tA$ acts on $X$ through the quotient $\tA\to A$)
such that the action of $\gm\subset\tA$ on $\ell$ is the tautological action by fiberwise dilations.

Moreover, one can take $\tA$ to be complement of the zero section in the total space of line bundle $\ell_A$ over $A$
(described after Definition~\ref{def:squarish}).
\end{proposition}
\begin{proof} For the `if' direction, we see that $[\ell]\in\Pic(X)(k)$ is fixed by the action of~$\tA$. Since $\tA$ acts on $X$,
and therefore on $\Pic(X)(k)$, through the quotient $\tA\to A$, the claim follows from Proposition~\ref{pr:SquarishFixed}.

For the `only if' direction, we can construct $\tA$ as the scheme of lifts of the action of $A$ on $X$ to $\ell$. Explicitly,
let $\tA$ be complement of the zero section in the total space of line bundle $\ell_A$. It is easy to see that~$\tA$ represents the functor that sends a test $k$-scheme $Z$ to the set of pairs
\[
    \Mor(Z,\tA)=\{(\phi,\eta)|\phi:Z\to A,\eta:\mu_\phi^*\ell\xrightarrow{\simeq}p_2^*\ell\}.
\]
Here $\mu_\phi$ is the composition
$Z\times X\xrightarrow{\phi\times\Id_X}A\times X\xrightarrow{\mu}X$.
It is easy to see that~$\tA$ is indeed a group scheme that is a central extension of $A$ by $\gm$, and that $\ell$ is equivariant under the
action of $\tA$.
\end{proof}

\begin{remarks}\label{rm:squarish}
\stepzero
\noindstep One can check that the extension \eqref{eq:extension} and the $\tA$-equivariant structure on $\ell$ satisfying the conditions of the proposition are unique up to a natural isomorphism.

\noindstep\label{rm:sq1comm}
Assume that $k$ is a field of characteristic zero and $A$ is a connected commutative group. Then $\tA$ is automatically commutative. Indeed, the commutator map on $\tA$ gives a bilinear pairing $A\times A\to\gm$. However, any bilinear pairing $A_1\times A_2\to\gm$, where $A_1$ and $A_2$ are connected commutative $k$-groups is trivial. Indeed, assume first that $A_1$ is affine. Then, by the Cartier duality, the datum of such a bilinear pairing is equivalent to the datum of a homomorphism from $A_2$ to the formal group $\Hom(A_1,\gm)$ (see~\cite[Exp.~$\mathrm{VII_B}$, (2.2.2)]{SGA3-1} or~\cite[(5.2)]{LaumonFourierGeneralized}). However, any morphism from a reduced connected variety to the formal variety $\Hom(A_1,\gm)$ is constant (that is, factors through a schematic point of $\Hom(A_1,\gm)$). This implies that any homomorphism from $A_2$ to $\Hom(A_1,\gm)$ must be trivial.
In the general case, we see that the maximal affine subgroup of $A_1$ is in the kernel of the pairing, so we may assume that $A_1$ is an abelian variety (by the Chevalley Theorem). In this case any morphism from $A_1\times A_2$ to $\gm$ factors through the projection to $A_2$. But any such bilinear pairing must be trivial.
\end{remarks}

\subsection{Cohomology of line bundles}
We keep the notation of the previous subsection. The following claim is a generalization of~\cite[Prop.~1]{ArinkinJacobians} (cf.~Observation (vii) in~\cite[Sect.~8]{MumfordAbelian}).
Recall that if $x$ is a $k$-rational point of $X$, then we have a morphism $\mu_x:A\to X$.

\begin{proposition}\label{pr:nonvanishcoh}
Suppose $\ell$ is a squarish line bundle over $X$, and $H^\bullet(X,\ell)\ne0$. Let $x\in X$ be a $k$-rational point. Then for some $d>0$
\[
    \mu_x^*\ell^{\otimes d}\simeq\cO_A.
\]
\end{proposition}
\begin{proof}
The proof essentially repeats that of~\cite[Prop.~1]{ArinkinJacobians}. Let $\tA$ be complement of the zero section in the total space of $\ell_A$. Proposition~\ref{pp:squarishequiv} gives a structure of a~group scheme on $\tA$ and an action of $\tA$ on $\ell$ such that $\gm\subset\tA$ acts tautologically by dilations. It is easy to see that $\ell_A\simeq\mu_x^*\ell$ (in particular, the isomorphism class of $\mu_x^*\ell$ is independent of the choice of $x$). Thus $\tA$ is a $\gm$-torsor over $A$, and the corresponding line bundle
\[
    \tA\times^{\gm}\mathbb{A}^1
\]
is isomorphic to $\mu_x^*\ell$.

The group scheme $\tA$ acts on $V:=H^\bullet(X,\ell)$. Note that $V$ is finite dimensional because $X$ is proper; put $d=\dim V$. We obtain an action of $\tA$ on $\det(V)$ such that $\gm$ acts as
\[
    a\mapsto a^d.
\]
The corresponding character splits the $d$-th multiple of the extension~\eqref{eq:extension}. This splitting gives a nowhere vanishing section of $\mu_x^*\ell^{\otimes d}$.
\end{proof}
\begin{remark}
Suppose the ground field $k$ is algebraically closed and $\tA$ is commutative. For instance, this holds if $k$ is an algebraically closed
field of characteristic zero and $A$ is connected and commutative, see Remark~\ref{rm:squarish}\eqref{rm:sq1comm}. We can then take
$V$ to be a one-dimensional subrepresentation $V\subset H^\bullet(X,\ell)$. Therefore, we can always guarantee that $d=1$ in this case (cf.~\cite[Prop.~1]{ArinkinJacobians}).
\end{remark}

\subsection{Cubish line bundles}
Let $\mu_A:A\times A\to A$ be the group multiplication, let  $\nu_A:A\times A\to A$ be the division morphism as in Section~\ref{sect:quasi-subgr}. Let $\Theta$ be a line bundle over $X$. Consider the product $A\times A\times X$ and denote by $p_{23}$ and $p_{13}$ its projections onto $A\times X$, by $p_{12}$ its projection onto $A\times A$. Consider the following four morphisms
\[
    \mu\circ(\mu_A\times\Id_X),\mu\circ p_{23},\mu\circ p_{13},p_3:A\times A\times X\to X,
\]

\begin{definition}\label{def:cubish} The line bundle $\Theta$ is \emph{cubish} (with respect to the action of $A$ on $X$) if the line bundle
\begin{equation}\label{eq:cubish}
(\mu\circ(\mu_A\times\Id_X))^*\Theta\otimes p_3^*\Theta\otimes(\mu\circ p_{23})^*\Theta^{-1}\otimes(\mu\circ p_{13})^*\Theta^{-1}
\end{equation}
over $A\times A\times X$ is isomorphic to the pullback of a line bundle over $A\times A$.
\end{definition}

\begin{example}\label{ex:ThmCube}
Let $X=A$ be an abelian variety acting on itself by translations. The Theorem of the Cube implies that any line bundle over $X$ is cubish; see~\cite[Thm.~8.4]{PolishchukAbelian}.
\end{example}

\begin{proposition}
A line bundle $\Theta$ is cubish if and only if the morphism $s'_\Theta:A\to\Pic(X)$ corresponding to the line bundle $\mu^*\Theta\otimes p_2^*\Theta^{-1}$ over $A\times X$ is a group homomorphism.
\end{proposition}
\begin{proof}
    The morphism $s'_\Theta$ is a group homomorphism if and only if the two induced morphisms $A\times A\to\Pic(X)$ coincide. These morphisms correspond to two line bundles over $A\times A\times X$. One just needs to check that the quotient of these line bundles is equal to~\eqref{eq:cubish} and use Proposition~\ref{pp:Cartesian}\eqref{pp:Cart2}.
\end{proof}

\begin{proposition}\label{pr:CubishFiberwise} Suppose $A$ is geometrically reduced. Then a line bundle $\Theta$ is cubish if and only if for any point $a\in A$, the line bundle $\mu_a^*\Theta\otimes p_a^*\Theta^{-1}$ over $X_{k(a)}$ is squarish with respect to the action
of the $k(a)$-group scheme $A_{k(a)}$. Here $\mu_a$ and $p_a$ are as in Proposition~\ref{pr:squarish reduced}.
\end{proposition}
\begin{proof}
    Assume that for any point $a\in A$, the line bundle $\mu_a^*\Theta\otimes p_a^*\Theta^{-1}$ over $X_{k(a)}$ is squarish. Then the restriction of the line bundle~\eqref{eq:cubish} to a fiber of $p_1:A\times A\times X\to A$ over any $a\in A$ is isomorphic to the pullback of a line bundle over $A_{k(a)}$. Thus the restriction of~\eqref{eq:cubish} to each fiber of the projection $A\times A\times X\to A\times A$ is trivial. Now it follows from Corollary~\ref{cor:see-saw} that $\Theta$ is cubish (note that $A\times A$ is reduced, since $A$ is geometrically reduced). The converse statement is left to the reader.
\end{proof}

Fix a cubish line bundle $\Theta$ over $X$. We call the line bundle
\[
    P_\Theta:=\mu^*\Theta\otimes p_2^*\Theta^{-1}
\]
over $A\times X$ the \emph{Poincar\'e line bundle} (corresponding to $\Theta$ and the action of $A$ on $X$).
More generally, we can fix a squarish line bundle $\theta$ over $X$ and consider
the \emph{shifted Poincar\'e line bundle}
\[
    P_{\Theta,\theta}:=P_\Theta\otimes p_2^*\theta.
\]
Consider the pushforward
\[
    Rp_{1,*}P_{\Theta,\theta}\in D^b_{coh}(A).
\]
Its support is a closed subset
\[
    \supp(Rp_{1,*}P_{\Theta,\theta})\subset A^{top}.
\]
Our goal is to estimate its codimension from below.

Let us now fix a rational point $x\in X(k)$ (assumed to exist), and consider the morphism $\mu_x:A\to X$. The pullback $\mu_x^*\Theta$ is a line bundle over $A$. To it, we can associate a quasi-subgroup $K_{\mu_x^*\Theta}\subset A^{top}$ defined in~\eqref{eq:shifts}. We claim that $\supp(Rp_{1,*}P_{\Theta,\theta})$ is contained in one coset of $K_{\mu_x^*\Theta}$, in the following sense:

\begin{proposition}\label{pp:suppinT}
For any point $x\in X(k)$,
\[
    \nu_A(\supp(Rp_{1,*}P_{\Theta,\theta})\times\supp(Rp_{1,*}P_{\Theta,\theta}))\subset K_{\mu_x^*\Theta},
\]
where $\nu_A$ is the division morphism.
\end{proposition}
\begin{proof}
By Corollary~\ref{co:finite order base change}, we may assume that $k$ is algebraically closed. Next, we may assume that $A$ is reduced, and thus geometrically reduced.

Consider the following line bundles over $A\times A$:
\begin{align*}
P_{\Theta,x}&=(\Id_A\times\mu_x)^*P_\Theta\\
P_{\Theta,\theta,x}&=(\Id_A\times\mu_x)^*P_{\Theta,\theta}\simeq P_{\Theta,x}\otimes(\mu_x\circ p_2)^*\theta,
\end{align*}
and let
\begin{align*}
K_{\Theta,x}&=\{a\in A: \exists\; d>0 \text{ such that }P_{\Theta,x}^{\otimes d}|_{a\times A}\simeq\cO_{a\times A}\}\subset A^{top}\\
K_{\Theta,\theta,x}&=\{a\in A: \exists\; d>0 \text{ such that }P_{\Theta,\theta,x}^{\otimes d}|_{a\times A}\simeq\cO_{a\times A}\}\subset A^{top}
\end{align*}
be their finite order loci (note that $K_{\Theta,x}=K_{\mu_x^*\Theta}$). (One can check that if $x$ changes, $P_{\Theta,x}$ and $P_{\Theta,\theta,x}$ get multiplied by the pullback
of a line bundle under $p_1:A\times A\to A$, so in fact $K_{\Theta,x}$ and $K_{\Theta,\theta,x}$ do not depend on $x$.)

Next, it follows from Proposition~\ref{pr:CubishFiberwise} and the obvious fact that a product of squarish line bundles is squarish, that for all $a\in A$ the line bundle $P_{\Theta,\theta}|_{a\times X}$ is squarish. Now by Proposition~\ref{pr:nonvanishcoh} we see that if $H^\bullet(X,P_{\Theta,\theta}|_{a\times X})\ne0$, then $a\in K_{\Theta,\theta,x}$. Using the base change, we get
\[
    \supp(Rp_{1,*}P_{\Theta,\theta})\subset K_{\Theta,\theta,x}.
\]
Finally, a direct calculation, using that $\theta$ is squarish, shows that
\[
    \nu_A\left(\left(K_{\Theta,\theta,x}\times K_{\Theta,\theta,x}\right)^{qc}\right)\subset K_{\Theta,x}=K_{\mu_x^*\Theta}.
\]
\end{proof}

\begin{corollary}\label{cor:support}
Let $y$ be a $k$-rational point of $\supp(Rp_{1,*}P_{\Theta,\theta})$ (assumed to exist). Then $\supp(Rp_{1,*}P_{\Theta,\theta})$ is a subset of the shift of $K_{\mu_x^*\Theta}$ by $y$.
\end{corollary}

\subsection{}\label{sect:stab}
Let $\delta(A)$ be the dimension of the maximal connected reduced affine subgroup scheme in $\spec\bar k\times A$, where $\bar k$ is an algebraic closure of $k$. Note that in Section~\ref{MainResults} we gave a slightly different definition of $\delta(A)$ under the additional assumption that the characteristic of $k$ is zero. According to the Chevalley theorem, the definitions agree.

If $A$ is commutative, then under certain mild assumptions, we can equivalently compute $\delta$ as the maximal dimension of stabilizers. For a schematic point $x\in X$, we denote the stabilizer of~$x$ by $\St_x\subset A_x:=  x\times A$.

\begin{lemma} \label{lm:allornone}
Either $\St_x$ is affine for all points $x\in X$ or it fails to be affine for all points $x\in X$.
\end{lemma}
\begin{proof} Using extensions of the ground field $k$, we see that it suffices to verify the claim for points in $X(k)$. Also,
we may assume that $\bar k=k$. Finally, we may assume that $X$ is integral.

Denote by $A'\subset A$ the kernel of the action of $A$ on $X$. Then $\St_x\supset A'$ for all $x\in X(k)$. In particular, if $A'$ is
not affine, then $\St_x$ fails to be affine for all $x\in X(k)$.

Suppose now that $A'$ is affine. Let $\cO_{X,x}$ be the local ring of $x\in X$, and let $\fm_{X,x}\subset\cO_{X,x}$ be the maximal ideal. The group $\St_x$ acts on $\cO_{X,x}$. The kernel of the induced representation of $\St_x$ on $\cO_{X,x}/\fm_{X,x}^N$ coincides with $A'$ for $N\gg0$. Hence $\St_x$ is affine, as claimed.
\end{proof}

\begin{proposition}\label{pp:delta and st}
Assume that $A$ is commutative and there is a point $x\in X$ such that the stabilizer $\St_x$ is affine. Then $\delta(A)=\max\{\dim(\St_x):x\in X\}$; moreover, it is enough to take the maximum over closed points.
\end{proposition}
\begin{proof}
By Lemma~\ref{lm:allornone}, the inequality $\delta(A)\ge\max\{\dim(\St_x):x\in X\}$ follows from definitions. Conversely, let
\[A^{aff}_{\bar k}\subset A_{\bar k}=\spec \bar k\times A\]
be the maximal affine reduced connected group subscheme; its existence is guaranteed by the Chevalley Theorem. The affine scheme
$A^{aff}_{\bar k}$ acts on the projective scheme $X_{\bar k}=\spec\bar k\times X$, and therefore by the Borel Theorem~\cite[Thm.~8.4]{BorelLinAlgGrps} the action has
a fixed point $\bar x\in X(\bar k)$. Let the closed point $x$ be the image of $\bar x:\spec\bar k\to X$, then
\[
    \dim(\St_x)=\dim(\St_{\bar x})\ge\dim A^{aff}_{\bar k}=\delta(A).
\]
\end{proof}

\begin{proposition}\label{pr:mainfiberwise}
Assume that there is a closed point $x\in X$ whose stabilizer $\St_x$ is finite.
Let $\Theta$ be an ample cubish line bundle over $X$; let $\theta$ be a squarish line bundle over $X$. Then
\[
    \dim\supp Rp_{1,*}P_{\Theta,\theta}\le\delta(A).
\]
\end{proposition}
\begin{proof}
We may assume that $k$ is algebraically closed. Let $x\in X(k)$ be a point with finite stabilizer. By Corollary~\ref{cor:support}
$\supp Rp_{1,*}P_{\Theta,\theta}$ is contained in a coset of $K_{\mu_x^*\Theta}$. By Proposition~\ref{pp:quasigroups} each component of $K_{\mu_x^*\Theta}$ is a coset of an algebraic subgroup $K^0_{\mu_x^*\Theta}$, and we see that
\[
    \dim\supp Rp_{1,*}P_{\Theta,\theta}\le\dim K^0_{\mu_x^*\Theta}.
\]
Thus we only need to show that $K^0_{\mu_x^*\Theta}$ is affine. By Corollary~\ref{cor:fiberestimate} this will follow if we show that
$\mu_x^*\Theta$ is an ample line bundle over $A$.

Now we prove that $\mu_x^*\Theta$ is ample. Note that the image $\mu_x(A)\subset X$ is the $A$-orbit of $x$; therefore the image is locally closed. Moreover, when viewed as a morphism onto its image,
\[
    \mu_x:A\to\mu_x(A)
\]
is quasi-finite, and therefore generically finite. Using the $A$-equivariance of this morphism, we see that it is in fact finite.
Hence $\mu_x^*\Theta$ is ample over $A$, as the pullback of an ample line bundle under a finite morphism (see e.g.~\cite[Prop.~5.1.12]{EGAII}).
\end{proof}

\section{Direct image of the Poincar\'e line bundle}
In this section, we discuss families of squarish and cubish line bundles, and define shifted Poincar\'e line bundles for
degenerate abelian schemes. Our main result is Proposition~\ref{pr:dirimpoinc}.

\subsection{Squarish line bundles in families}
Let us now consider a relative version of the framework from the previous section. Let $k$ be a field, let $B$ be a base scheme
that we suppose to be locally of finite type over $k$. Let $\pi_X:X\to B$ be a flat projective morphism whose geometric fibers
are non-empty, reduced, and connected. Let $\pi_A:A\to B$ be a group scheme of finite type over $B$. Finally, let
\[
    \mu:A\times_BX\to X
\]
be an action of $A$ on $X$.

\begin{definition} Let $\ell$ be a line bundle over $X$. We say that $\ell$ is \emph{(relatively) squarish} if the line bundle
\[\ell':=\mu^*\ell\otimes p_2^*\ell^{-1}\]
over $A\times_B X$ is isomorphic to the pullback of a line bundle over $A$.

We say that $\ell$ is \emph{(relatively) cubish} if the line bundle $\ell'$ is squarish with respect to the action of the $A$-group scheme $A\times_BA$ on $A\times_B X$.
\end{definition}

Unwinding the definition of the relatively cubish line bundle, we see that if $B$ is the spectrum of a field, it is equivalent to Definition~\ref{def:cubish}.

\begin{remark} The line bundle $\ell$ yields a section $s_\ell:B\to\Pic(X/B)$ of the relative Picard space. The action of $A$ on $X$
induces an action
\[\mu_{\Pic}:A\times_B\Pic(X/B)\to\Pic(X/B).\]
Similarly to Proposition~\ref{pr:SquarishFixed}, $\ell$ is squarish if and only if $s_\ell$ is invariant under this action; that is, the morphism
\[
    \mu_{\Pic}\circ(\Id_A\times(s_\ell\circ\pi_A)):A\to\Pic(X/B)
\]
is `fiberwise constant'. Similarly, $\ell$ is cubish if and only if
the morphism is `fiberwise affine linear'.
\end{remark}

\begin{proposition}\label{pr:sq}
\stepzero\noindstep\label{sq:a} Suppose $B$ is reduced. A line bundle $\ell$ is relatively squarish (resp. cubish) if and only if the restriction $\ell_b$ of $\ell$ to ${X_b}$ is squarish (resp. cubish) for every schematic point $b\in B$.

\noindstep\label{sq:b} Suppose that $A$ is flat over $B$, and that one of the following two conditions is satisfied:
either $\ell$ is fiberwise numerically trivial, or the fibers of $\pi_A$ are connected.

Under these assumptions, the set of points $b\in B$ such that $\ell_b$ is squarish (resp.\ cubish) is specialization-closed. Equivalently,
the set is a possibly infinite union of closed sets.
\end{proposition}
\begin{proof}
Let us prove~\eqref{sq:a}. By definition, $\ell$ is squarish if and only if the line bundle~$\ell'$ over $A\times_B X$
is isomorphic to the pullback of a line bundle over $A$. Let ${s_\ell':A\to\Pic(A\times_BX/A)}$ be the corresponding section of the
relative Picard space; by Proposition~\ref{pp:Cartesian}\eqref{pp:Cart2} $\ell$ is squarish if and only if $s_\ell'$ is the zero section. Since the zero section is a locally
closed embedding by Proposition~\ref{pr:zerosection}, this condition can be verified on fibers of the morphism $A\to B$, provided $B$ is reduced.

The argument for cubish line bundles is similar: $\ell$ is cubish if and only if a line bundle $\ell''$ over $A\times_BA\times_BX$
is isomorphic to the pullback of a line bundle over $A\times_B A$. This condition can be checked on fibers of the projection
$A\times_BA\to B$.

Now let us prove~\eqref{sq:b}. Under our assumptions, the morphism $s_\ell'$ factors through $\Pic^\tau(A\times_BX/A)$ (note that since $\pi_A:A\to B$ is a group scheme, its fibers are connected if and only if they are geometrically connected). By Proposition~\ref{pr:zerosection} the zero section is closed in $\Pic^\tau(A\times_BX/A)$, so we see that the preimage $Z\subset A$ of
the zero section under $s_\ell'$ is a closed subscheme in $A$. Given $b\in B$, we see that $\ell_b$ is squarish if and only if $\pi_A^{-1}(b)$ is contained in $Z$ (scheme theoretically). The set of such $b$ is specialization-closed because $\pi_A$ is flat.

The argument for cubish bundles is similar.
\end{proof}

\begin{remark} In the situation of Proposition~\ref{pr:sq}\eqref{sq:b}, suppose that $\pi_A$ has geometrically reduced fibers.
Then $\ell_b$ is squarish if and only if $b\in B-\pi_A(A-Z)$ (where~$Z$ is defined in the proof of the proposition). Thus, the
locus of $b\in B$ such that~$\ell_b$ is squarish is closed, and not just specialization-closed;
a similar argument works for cubish bundles.
\end{remark}

Finally, note the behavior of the defect $\delta$ in families.

\begin{proposition}\label{pp:delta semicontinuous}
Suppose that $A$ is commutative and that for every point $b\in B$, there is a point $x\in X_b$ in the fiber over $b$ whose stabilizer is affine.
Then the function
\[
    b\mapsto\delta(A_b):B\to\Z
\]
is upper-semicontinuous.
\end{proposition}
\begin{proof}
Fix $k$, and let us show that the set
\[\{b\in B:\delta(A_b)\ge k\}\subset B\]
is closed.
Put
\[
    Y:=\{x\in X:\dim(\St_x)\ge k\}.
\]
Here $\St_x$ is the stabilizer of $x\in X$ (cf.~Section~\ref{sect:stab}).
Then $Y$ is closed by \cite[Thm.~13.1.3]{EGAIV-3} applied to the natural morphism $A\times_BX\to X\times_BX$. Further,
\[
    \{b\in B:\delta(A_b)\ge k\}=\pi_X(Y)
\]
by Proposition~\ref{pp:delta and st}. Since $\pi_X$ is proper, the proposition follows.
\end{proof}

\subsection{Line bundles over degenerate abelian schemes}
Let us now impose additional restrictions on the collection $(\pi_X:X\to B,\pi_A:A\to B,\mu:A\times_B X\to X)$:
we require it to be a degenerate abelian scheme in the sense of Definition~\ref{def:DAS}. Replacing $B$ by a Zariski cover,
we can find a relatively ample line bundle $\Theta$ over $\pi_X:X\to B$. Let $\theta$ be a fiberwise numerically trivial line bundle over $\pi_X:X\to B$. Define the \emph{shifted Poincar\'e line bundle} over $A\times_BX$ as
\[
    P_{\Theta,\theta}:=\mu^*\Theta\otimes p_2^*\Theta^{-1}\otimes p_2^*\theta
\]

\begin{proposition}\label{pr:dirimpoinc}
Assume that the fibers of $\pi_A$ are connected. Then $Rp_{1,*}P_{\Theta,\theta}$ is supported in cohomological degree $g$.
\end{proposition}
\begin{proof}
First of all, we claim that every line bundle over $X$ is relatively (and thus by Proposition~\ref{pr:sq}\eqref{sq:a} fiberwise) cubish. Indeed, every line bundle is cubish over the generic fiber because this fiber is a torsor over an abelian variety (this follows from the Theorem of the Cube, see Example~\ref{ex:ThmCube}). It remains to use Proposition~\ref{pr:sq}\eqref{sq:b}. Thus $\Theta$ is relatively (and fiberwise) cubish.

Next, combining Proposition~\ref{pp:abelian variety} and Lemma~\ref{lm:finite order extension}, we see that $\theta$ is squarish over the generic fiber. Thus it is relatively (and fiberwise) squarish. Consider the closed subset $B':=\supp Rp_{1,*}P_{\Theta,\theta}$.

Next, by Definition~\ref{def:DAS}\eqref{DAS:FiniteStab} every fiber $X_b$ of $\pi_X:X\to B$ has a point $x$ with finite stabilizer in $A\times_Bx$. The set of such $x$ is open in $X_b$ (see the proof of Proposition~\ref{pp:delta semicontinuous}), so we may assume that $x$ is a closed point. Thus we can apply Proposition~\ref{pr:mainfiberwise} to the fiber. Combining  this proposition and the estimates~\eqref{eq:condition}, we see that  $\dim B'\le\dim B$ so that $\codim B'\ge g$. Since $P_{\Theta,\theta} $ is flat over $A$, we see that $Rp^i_{1,*}P_{\Theta,\theta}=0$ for $i<g$.
\end{proof}

\begin{remark}
It is easy to see from the proof that $\supp Rp_{1,*}P_{\Theta,\theta}$ is Cohen--Macaulay of pure codimension $g$.
\end{remark}

\section{Proof of Theorem~\ref{th:PartialFourier}}\label{ProofOfMainThm}
We keep the notation of Theorem~\ref{th:PartialFourier}. In particular, $k$ is a field of characteristic zero.
Since the statement of Theorem~\ref{th:PartialFourier} is local over the base $B$, we can assume that $B$ is connected and there exists a relatively ample line bundle over $\pi_X:X\to B$ without losing generality. Let us fix such a line bundle $\Theta$.

Also, by~\cite[Exp.~$\mathrm{VI_B}$, Thm.~3.10]{SGA3-1}, the union of neutral connected components of the fibers of the group scheme $\pi_A:A\to B$ is an open subscheme of $A$. Replacing $A$ by this subscheme, we may assume that the fibers of $\pi_A$ are connected, and thus geometrically connected because $\pi_A:A\to B$ is a group scheme.

Let $P_\Theta:=\mu^*\Theta\otimes p_2^*\Theta^{-1}$ be the Poincar\'e line bundle over $A\times_BX$ given by $\Theta$. It gives rise to a morphism $\phi:A\to\Pic(X/B)$.

\subsection{} The first step in the proof of Theorem~\ref{th:PartialFourier} is the following
\begin{proposition}\label{pr:unramified}
The morphism $\phi:A\to\Pic(X/B)$ is unramified.
\end{proposition}

Let us make some observations before proving Proposition~\ref{pr:unramified}.

Denote by $\zeta_A:B\to A$ the zero section. Let $\fa:=\Lie(A/B)=\zeta_A^*\cT(A/B)$ be the Lie vector bundle of the smooth group scheme $\pi_A:A\to B$; it is a vector bundle over $B$. We have $\cT(A/B)=\pi_A^*\fa$. Let $\omega_{A/B}$ be the relative dualizing line bundle of $\pi_A$. We have $\omega_{A/B}=\pi_A^*(\wedge^g\fa)^{-1}$. Also, by Definition~\ref{def:DAS}\eqref{DAS:gentorsor}, there is an open dense subset $B_0\subset B$ such that $X_0:=\pi_X^{-1}(B_0)$ is a torsor over the abelian scheme $A_0:=\pi_A^{-1}(X_0)$; set $\pi_{X_0}:=\pi_X|_{X_0}$. The infinitesimal action of $\Lie(A_0/B_0)$ on $X_0$ gives an isomorphism
\[
    \cT(X_0/B_0)=\pi_{X_0}^*(\Lie(A_0/B_0))=\pi_X^*\fa|_{X_0}.
\]
On the other hand, by Definition~\ref{def:DAS}\eqref{DAS:RelDualLB}, the relative dualizing line bundle $\omega_{X/B}$ of
$\pi_X:X\to B$ is of the form $\pi_X^*\ell_1$ for some line bundle $\ell_1$ over $B$. Let us fix such $\ell_1$ and an isomorphism. This provides an isomorphism
\begin{equation}\label{eq:TangentToTorsor}
    \pi_X^*\ell_1|_{X_0}=\omega_{X_0/B_0}=\pi_X^*(\wedge^g\fa)^{-1}|_{X_0}.
\end{equation}
Applying the pushforward $\pi_{X_0,*}$ to this isomorphism and using Lemma~\ref{lm:dirimagestructsheaf}, we obtain a canonical isomorphism.
\begin{equation}\label{eq:iso}
\ell_1|_{B_0}=(\wedge^g\fa)^{-1}|_{B_0}.
\end{equation}

\begin{lemma}\label{lm:can}
The isomorphism~\eqref{eq:iso} extends to an isomorphism of line bundles $\ell_1=(\wedge^g\fa)^{-1}$, so that
$\omega_{X/B}=\pi_X^*(\wedge^g\fa)^{-1}$.
\end{lemma}
\begin{proof} Let $X'\subset X$ be the locus of points $x\in X$ such that $\pi_X$ is smooth at $x$ and the stabilizer of $x$
in $A_{\pi_X(x)}$ is finite. Since the dimensions of the stabilizers are upper-semicontinuous (see the proof of Proposition~\ref{pp:delta semicontinuous}), $X'$ is open. Also, $\pi_X(X')=B$. Indeed, for any closed point $b\in B$, the smooth locus of the fiber~$X_b$ is open and dense, while the locus of points with finite stabilizers is open and non-empty. Denote the restriction of $\pi_X$ to $X'$ by $\pi_{X'}$.

As above, over $X'$ the action $\mu$ gives an isomorphism
\[
    \Lie(\mu):\pi_X^*\fa|_{X'}\to\cT(X'/B).
\]
It induces an isomorphism
\[
    \pi_{X'}^*(\wedge^g\fa)^{-1}=\omega_{X'/B}=\pi_{X'}^*\ell_1
\]
whose restriction to $X_0$ coincides with~\eqref{eq:TangentToTorsor}. Since $X'\to B$ is a smooth surjective
morphism, this isomorphism descends to $B$. In more detail, we need to check that two induced isomorphisms on $X'\times_BX'$ coincide, which is enough to check on $X'\times_{B_0}X'$, but over $B_0$ our isomorphism descends to~\eqref{eq:iso}.
\end{proof}

We fix an isomorphism provided by the lemma.

\begin{proof}[Proof of Proposition~\ref{pr:unramified}] $\phi$ is unramified if and only if the canonical morphism
\begin{equation}\label{eq:cotangentmorphism}
   \phi^*\Omega_{\Pic(X/B)/B}\to\Omega_{A/B}
\end{equation}
is surjective (see~\cite[Cor.~17.2.2]{EGAIV.4}). Consider also the dual morphism
\[
    d\phi:\HOM_A(\Omega_{A/B},\cO_A)\to\HOM_A(\phi^*\Omega_{\Pic(X/B)/B},\cO_A).
\]
Note that
\[
    \HOM_A(\Omega_{A/B},\cO_A)=\pi_A^*\Lie(A/B)=\pi_A^*\fa.
\]
It follows from the deformation theory (see for instance~\cite[Sect.~8.4, Thm.~1(a)]{BLR_NeronModels}) that
\[
   \HOM_{\Pic(X/B)}(\Omega_{\Pic(X/B)/B},\cO_{\Pic(X/B)})=\pi_{\Pic}^*R^1\pi_{X,*}\cO_X,
\]
where $\pi_{\Pic}:\Pic(X/B)\to B$ is the projection. Thus we have
\[
    d\phi:\pi_A^*\fa\to\pi_A^*R^1\pi_{X,*}\cO_X.
\]
In general, injectivity of $d\phi$ does not imply surjectivity of~\eqref{eq:cotangentmorphism}. To circumvent this difficulty, consider a
geometric point $s\to B$. It is enough to show that for every such point the corresponding base change of~\eqref{eq:cotangentmorphism} is surjective. After such a base change, both the source and the target of~\eqref{eq:cotangentmorphism} become locally free, because they are homogeneous with respect to the action of $A_s$ on itself by translation. The dual morphism is given by
\[
    d\phi_s:\pi_{A_s}^*\fa_s\to\pi_{A_s}^*R^1\pi_{X_s,*}\cO_{X_s},
\]
where $\pi_{A_s}$ and $\pi_{X_s}$ are the projections.
We need to show that this morphism of vector bundles is an embedding, that is, that it identifies $\pi_{A_s}^*\fa_s$
with a subbundle of $\pi_{A_s}^*R^1\pi_{X_s,*}\cO_{X_s}$.

\begin{lemma}
The morphism $d\phi_s$ is an embedding of vector bundles for all geometric points $s\to B$.
\end{lemma}
\begin{proof}
Lemma~\ref{lm:can} provides a morphism
\begin{equation*}
    R^g\pi_{X,*}\cO_X= R^g\pi_{X,*}(\omega_{X/B}\otimes\pi_X^*(\wedge^g\fa))=
    R^g\pi_{X,*}\omega_{X/B}\otimes\wedge^g\fa\xrightarrow{\mathrm{tr}\otimes\Id}\wedge^g\fa.
\end{equation*}
Consider the composition
\begin{equation}\label{eq:composition}
    \pi_A^*(\wedge^g\fa)\xrightarrow{\wedge^gd\phi}
    \pi_A^*(\wedge^gR^1\pi_{X,*}\cO_X)\xrightarrow{\cup^g}
    \pi_A^*(R^g\pi_{X,*}\cO_X)\to\pi_A^*(\wedge^g\fa).
\end{equation}

The composition~\eqref{eq:composition} is a multiplication by a function $f$ on $A$.

On the other hand, for a geometric point $s\to B$ we can consider the composition
\begin{equation}
    \pi_{A_s}^*(\wedge^g\fa_s)\xrightarrow{\wedge^gd\phi_s}
    \pi_{A_s}^*(\wedge^gR^1\pi_{X_s,*}\cO_{X_s})\xrightarrow{\cup^g}
    \pi_{A_s}^*(R^g\pi_{X_s,*}\cO_{X_s})\to\pi_{A_s}^*(\wedge^g\fa_s).
\end{equation}
One checks that this composition is given by multiplication by $f(s)$. Thus to prove the lemma, it is enough to show that $f$ has no zeros on $A$.

It is enough to check that $f$ is a non-zero constant on $\pi_A^{-1}(B_0)$. To this end, since $B$ is assumed to be connected, it is enough to check that $f$ takes non-zero integer values on $\pi_A^{-1}(B_0)$. But for $b\in B_0$, the fiber $A_b$ of $\pi_A$ is an abelian variety, and $f$ is equal to the top cup-power of the first Chern class of $\Theta_b$:
\[
    \cup^g\mathrm{ch}_1(\Theta_b)\in H^{2g}(A_b,\Z)=\Z.
\]
This is well-known and follows, for example, from Consequence~(i) of the first proposition in~\cite[Sect.~9]{MumfordAbelian}. Technically, in \emph{ibid.\/} it is assumed that $k$ is the field of complex numbers, but the statement is easily reduced to this special case
(first reduce to the case when $k$ is of finite transcendence degree over the field of rational numbers, and then embed $k$
into the field of complex numbers).
\end{proof}
This completes the proof of the proposition.
\end{proof}

\subsection{Proof of smoothness of $\Pic^\tau(X/B)$}
Recall that $P_\Theta$ is the Poincar\'e bundle over $A\times_BX$ constructed from a relatively ample line bundle $\Theta$ over $X$.
Let $\tB\subset A$ be the support of $\HOM(Rp_{1,*}P_\Theta^\vee,\cO_A)$. By Propositions~\ref{pr:trivlocus}\eqref{pr:trivlocus1} and~\ref{pp:Cartesian}\eqref{pp:Cart2} the following diagram is Cartesian
\begin{equation*}
\begin{CD}
\tB @>\pi_{\tB}>> B\\
@V\iota_{\tB}VV @VV\iota'V\\
A @>-\phi>> \Pic^\tau(X/B).
\end{CD}
\end{equation*}
Here $\pi_{\tB}=\pi_A|_{\tB}$, $\iota_{\tB}$ is the closed embedding, the right vertical arrow is the zero section, and $-\phi$ is the composition of $\phi$ and the group inversion.
\begin{lemma}
The morphism $\pi_{\tB}:\tB\to B$ is unramified.
\end{lemma}
\begin{proof}
By Proposition~\ref{pr:unramified}, $-\phi$ is unramified, and the claim follows by base change.
\end{proof}

Clearly, $\zeta_A(B)\subset\tB$. Thus we have
\begin{corollary}\label{cor:component}
$\zeta_A(B)$ is a connected component of $\tB$.
\end{corollary}

\begin{proposition}\label{pr:DirectImageStructure}
We have
\[
    R\pi_{X,*}\cO_X\simeq\wedge^{\bullet}\fa:=\bigoplus_{i=0}^g\wedge^i\fa[-i],
\]
where $[-i]$ is the cohomological shift.
\end{proposition}
\begin{proof}
Recall that by Definition~\ref{def:DAS}\eqref{DAS:RelDualLB}, the relative dualizing line bundle $\omega_{X/B}$ of $\pi_X$
is of the form $\pi_X^*\ell_1$ for a line bundle $\ell_1$ over $B$. (By Lemma~\ref{lm:can}, $\ell_1=(\wedge^g\fa)^{-1}$,
but we do not need this fact here.)
By the Serre duality
\[
    Rp_{1,*}P_\Theta[g]=R\HOM(Rp_{1,*}P_\Theta^\vee,\cO_A)\otimes\pi_A^*\ell^\vee_1.
\]
Now Proposition~\ref{pr:dirimpoinc} and Proposition~\ref{pr:trivlocus}\eqref{pr:trivlocus2} show that
\[
    R\HOM(Rp_{1,*}P_\Theta^\vee,\cO_A)\otimes\pi_A^*\ell^\vee_1=\HOM(Rp_{1,*}P_\Theta^\vee,\cO_A)\otimes\pi_A^*\ell^\vee_1=\iota_{\tB,*}\ell
\]
for some line bundle $\ell$ over $\tB$. Thus
\[
    Rp_{1,*}P_\Theta=\iota_{\tB,*}\ell[-g].
\]
By base change
\[
    R\pi_{X,*}\cO_X=L\zeta_A^*Rp_{1,*}P_\Theta=L\zeta_A^*(\iota_{\tB,*}\ell)[-g].
\]
Let $\ell_0:=\ell|_{\zeta_A(B)}$. Using Corollary~\ref{cor:component} and the Koszul resolution of $\zeta_A(B)\subset A$, we get
\[
    R\pi_{X,*}\cO_X=\bigoplus_{i=0}^g\ell_0\otimes\wedge^i\fa^*[i-g].
\]
Taking zeroth cohomology and using Lemma~\ref{lm:dirimagestructsheaf}, we see that $\ell_0=\wedge^g\fa$, which implies the statement.
\end{proof}

\begin{corollary} \label{co:phiisetale}
The morphism $\phi:A\to\Pic(X/B)$ is \'etale.
\end{corollary}
\begin{proof}
By Proposition~\ref{pr:unramified}, $\phi$ is unramified. It follows from Proposition~\ref{pr:DirectImageStructure} and deformation theory that for every geometric point $\psi:s\to\Pic(X/B)$ the relative tangent space
\[
    \cT_s(\Pic(X/B)/B):=\{\tilde\psi\in\Mor_B(s\times\spec\Z[\epsilon]/\epsilon^2,\Pic(X/B))\,:\;\tilde\psi|s=\psi\}.
\]
is $g$-dimensional (cf.~\cite[Sect.~8.4, Thm.~1(a)]{BLR_NeronModels}). It remains to use the following general statement.
\end{proof}

\begin{lemma}
Let $B$ be as before; let $Z_1$ and $Z_2$ be algebraic spaces locally of finite type over $B$. Let $\rho:Z_1\to Z_2$ be an unramified $B$-morphism. Further, assume that $Z_1$ is smooth of dimension $g$ over $B$ and for any geometric point $s\to Z_2$ the tangent space $\cT_s(Z_2/B)$ is $g$-dimensional. Then $\rho$ is \'etale.
\end{lemma}
\begin{proof}
The statement is \'etale local, so we can assume that $Z_1$ and $Z_2$ are schemes of finite type over $B$. Since $\rho$ is unramified, it can be written locally as a composition of a closed embedding and an \'etale morphism (see~\cite[Cor.~18.4.7]{EGAIV.4}). Thus we can assume that $\rho$ is a closed embedding. For closed embeddings, the statement is easy to prove.
\end{proof}

\begin{corollary}
There is an open subspace $\Pic^0(X/B)\subset\Pic^\tau(X/B)$ representing the functor
\begin{multline*}
    \Pic^0(X/B)(S)=\\ \{f\in\Pic(X/B)(S): \text{for all geometric points $s\to S$, } f(s)\in\Pic^0(X_s)\}.
\end{multline*}
Moreover, $\Pic^0(X/B)$ is smooth over $B$.
\end{corollary}
\begin{proof}
We claim that $\phi(A)=\Pic^0(X/B)$. Indeed, over every geometric point $s\to B$,
$\phi$ induces a morphism of fibers $\phi_s:A_s\to\Pic(X_s)$. Since $\phi_s$ is an \'etale morphism of group schemes (by Corollary~\ref{co:phiisetale}), the image of a connected component is again a connected component. Thus $\phi(A_s)=\Pic^0(X_s)$, and the statement follows.

Using Corollary~\ref{co:phiisetale} again, we see that $\phi(A)$ is open and smooth over $B$.
\end{proof}

\begin{proof}[Proof of Theorem~\ref{th:PartialFourier}\eqref{th:PartFourier1}]
For a positive integer $d$, consider the morphism \[\hat d:\Pic^\tau(X/B)\to\Pic^\tau(X/B),\] sending a line bundle $\ell$ to $\ell^{\otimes d}$. Note that $\hat d$ is \'etale because its differential is the multiplication by $d$.
(See~\cite[Thm.~2.5]{FGAVI} in the case when $\Pic(X/B)$ exists as a~scheme and note that the proof carries over to the case of algebraic spaces.) Hence $\hat d^{-1}(\Pic^0(X/B))$ is smooth over $\Pic^0(X/B)$ and thus over $B$; also,
It remains to notice that \[\Pic^\tau(X/B)=\bigcup_{d>0}\hat d^{-1}(\Pic^0(X/B))\]
is a union of open subspaces.

Finally, since $\cPic^\tau(X/B)$ is a $\gm$-gerbe over $\Pic^\tau(X/B)$, $\cPic^\tau(X/B)$ is smooth over $B$ as well.
\end{proof}

\subsection{End of proof of Theorem~\ref{th:PartialFourier}}
Our first goal is to prove that $Rp_{1,*}\cL^\vee$ is supported in cohomological degree $g$. This statement is local over $\cPic^\tau(X/B)$, so let us consider a closed point $x\in\cPic^\tau(X/B)$. Using the first part of the theorem and performing an \'etale base change, we may assume that there exists a section $\tilde x:B\to\cPic^\tau(X/B)$ passing through $x$. This section corresponds to a line bundle~$\theta$ over $X$. Consider the shifted Poincar\'e line bundle $P_{\Theta,\theta}:=P_\Theta\otimes p_2^*\theta$. The dual line bundle $P_{\Theta,\theta}^\vee$ gives rise to a morphism $\phi':A\to\cPic^\tau(X/B)$. Let $\phi''$ be the composition of $\phi'$ with the projection to $\Pic^\tau(X/B)$. Then $\phi''$ is the composition of group scheme inversion, $\phi$, and the shift by $\theta$, so it is \'etale. It follows that $\phi''$ is flat and open. Since $\cPic^\tau(X/B)$, is a $\gm$-gerbe over $\Pic^\tau(X/B)$,  we see that $\phi'$ is flat and open; we denote its image by $U$ and note that $x\in U$.

We need to show that $Rp_{1,*}\cL'$ is supported in degree $g$, where $\cL':=\cL^\vee|_{U\times_BX}$. Consider the Cartesian diagram
\[
\begin{CD}
A\times_BX @>\phi''\times_B\Id_X>>U\times_BX\\
@V{p_1}VV @V{p_1}VV\\
A @>\phi''>>U.
\end{CD}
\]
By definition of $\phi''$ we have $(\phi''\times_B\Id_X)^*\cL'\simeq P_{\Theta,\theta}\otimes p_1^*\ell_0$ for some line bundle~$\ell_0$ over $A$. By base change and Proposition~\ref{pr:dirimpoinc}, $(\phi'')^*(Rp_{1,*}\cL')$ is supported in cohomological degree $g$. It remains to note that $\phi''$ is flat and surjective.

We see that $Rp_{1,*}\cL^\vee$ is supported in degree $g$. As before, write $\omega_{X/B}=\pi_X^*\ell_1$. By Serre duality
\[
    R\HOM(Rp_{1,*}\cL,\cO_{\cPic^\tau(X/B)})=(Rp_{1,*}\cL^\vee\otimes\pi_A^*\ell_1)[g]
\]
is supported in cohomological degree zero. Recall from Corollary~\ref{cor:Funiv} the sheaf
\[
    \cF_{univ}=\HOM(Rp_{1,*}\cL,\cO_{\cPic^\tau(X/B)})=R\HOM(Rp_{1,*}\cL,\cO_{\cPic^\tau(X/B)}).
\]
We have
\[
    Rp_{1,*}\cL=R\HOM(\cF_{univ},\cO_{\cPic^\tau(X/B)}).
\]
Combining Corollary~\ref{cor:Funiv}, the fact that $B\times {\mathrm B}(\gm)$ is a complete intersection of codimension $g$ in
$\cPic^\tau(X/B)$, and Koszul resolution, we see that $Rp_{1,*}\cL$ is of the form $\iota_*\ell'[-g]$, where $\ell'$ is a line bundle over $B\times {\mathrm B}(\gm)$. It is easy to see that $\gm$ acts according to the character $\lambda\mapsto\lambda$, so $\ell'\simeq\ell^{(1)}$ for a line bundle $\ell$ over $B$.

The proof of Theorem~\ref{th:PartialFourier} is complete.\hfill\qed

\section{Integrable systems}\label{IntegrableSystems}
In this Section we will prove Theorem~\ref{th:IntegrableSystems}. We will use the notation from the statement of the theorem.

\subsection{Relative tangent sheaf on the big fppf site}
Let us review the properties of the relative tangent sheaf; we follow \cite[Exp.~II]{SGA3-1}.
Let $T$ be a locally Noetherian scheme and let $\pi:Z\to T$ be a morphism locally of finite type (but not necessarily smooth).

The relative cotangent sheaf $\Omega_{Z/T}$ is a coherent sheaf on $Z$. Its dual is the relative tangent sheaf
\[
    \cT(Z/T):=\HOM(\Omega_{Z/T},\cO_Z)
\]
of $Z$ over
$T$. Equivalently, $\cT(Z/T)$ can be described as the sheaf of $\pi^{-1}(\cO_T)$-linear derivations of $\cO_Z$, or as the sheaf
of automorphisms of the scheme $Z\times(\spec\Z[\epsilon]/\epsilon^2)$ over $T\times (\spec\Z[\epsilon]/\epsilon^2)$ that restrict to
identity on $Z\subset Z\times(\spec\Z[\epsilon]/\epsilon^2)$.
Clearly, $\cT(Z/T)$ is a coherent sheaf on $Z$.

Since $\pi$ is not assumed to be smooth, it is important to upgrade the definition of the relative tangent sheaf and define its sections
over arbitrary $Z$-schemes. Given a test scheme $\psi:S\to Z$ over $Z$, put
\begin{align*}
    \cT_{fppf}(Z/T)(\psi)&=\{\tilde\psi\in\Mor_T(S\times(\spec\Z[\epsilon]/\epsilon^2),Z):\;\tilde\psi|_S=\psi\}\\
    &=\Hom_S(\psi^*\Omega_{Z/T},\cO_S).
\end{align*}
As $\psi:S\to Z$ varies, we obtain a sheaf of $\cO$-modules $\cT_{fppf}(Z/T)$ on the big fppf site of schemes over $Z$.

\begin{remark} The sheaf $\cT_{fppf}(Z/T)$ is representable by a $Z$-scheme: the total space of the relative tangent sheaf.
\end{remark}

For a quasicoherent sheaf $\cF$ on $Z$, we denote by $\cF_{fppf}$ the corresponding sheaf on the big fppf site; thus,
\[\cF_{fppf}(\psi)=\Gamma(S,\psi^*\cF)\] for
a morphism of schemes $\psi:S\to Z$. If $\pi$ is smooth, then $\cT_{fppf}(Z/T)=\cT(Z/T)_{fppf},$
and $\cT_{fppf}(Z/T)$ is a locally free $\cO$-module of finite rank. In general, we have a morphism
\[\cT(Z/T)_{fppf}\to\cT_{fppf}(Z/T),\]
but $\cT_{fppf}(Z/T)$ cannot be reconstructed from $\cT(Z/T)$.

Consider now the direct image $\pi_*\cT_{fppf}(Z/T)$. It is a sheaf on the big fppf site of schemes over $T$. Explicitly,
given a morphism of schemes $\psi:V\to T$, put $S:=Z\times_TV$, and consider the
natural diagram
\[\xymatrix{S\ar[r]^{\psi_Z}\ar[d]&Z\ar[d]^{\pi}\\V\ar[r]^{\psi}&T.}\]
The natural isomorphism $\Omega_{S/V}\simeq\psi_Z^*\Omega_{Z/T}$ induces a morphism
\begin{equation}\label{eq:pullbackreltan}
\psi_Z^*\cT(Z/T)\to\cT(S/V).
\end{equation}
If $\psi$ is flat, the morphism \eqref{eq:pullbackreltan} is an isomorphism; in general it need not be the case. By definition, we have
\[\pi_*\cT_{fppf}(Z/T)(\psi)=\cT_{fppf}(Z/T)(\psi_Z)=\Gamma(S,\cT(S/V)).\]
Note that $\pi_*\cT_{fppf}(Z/T)$ is a sheaf of $\cO_T$-Lie algebras.

Similar approach applies to the Lie algebra of a group scheme. Namely, given a~group scheme $G\to T$ locally of finite type, we let $\Lie(G/T)$ be the restriction of $\cT_{fppf}(G/T)$ to the unit section. Thus,
for a morphism $\psi:S\to T$, we have
\[
    \Lie(G/T)(\psi)=\{\tilde\psi\in Mor_T(S\times(\spec(\Z[\epsilon]/\epsilon^2),G):\tilde\psi|_S=\zeta_G\circ\psi\}.
\]
Here $\zeta_G:T\to G$ is the unit section. The sheaf $\Lie(G/T)$ carries a natural structure of a sheaf of $\cO_T$-Lie algebras.

We need a general
\begin{proposition}\label{pr:SmoothGrSch}
Let $T$ be a scheme of characteristic zero (that is, a scheme over $\mathbb Q$ not necessarily of finite type) and let $G\to T$ be a group scheme locally of finite type over $T$ such that $\Lie(G/T)=\cF_{fppf}$ for a locally free coherent
$\cO_T$-module $\cF$. Then $G\to T$ is smooth along the unit section.
\end{proposition}
\begin{proof}
Denote the unit section by $\zeta_G:T\to G$. By~\cite[Exp.~$\mathrm{VI_B}$, Prop.~1.6]{SGA3-1}), it is enough to check that the sheaf $\zeta_G^*\Omega_{G/T}$ is locally free. Let us prove that
\[
    \zeta_G^*\Omega_{G/T}\simeq\cF^\vee:=\HOM(\cF,\cO_T).
\]
By the definition of $\Lie(G/T)$, for any $T$-scheme $\psi:S\to T$ there are functorial isomorphisms
\[
    \Hom_S(\psi^*\zeta_G^*\Omega_{G/T},\cO_{S})\simeq \Gamma(S,\psi^*\cF)=\Hom_S(\psi^*\cF^\vee,\cO_{S}).
\]
It remains to use Lemma~\ref{lm:fppfdual} below.
\end{proof}

\begin{lemma}\label{lm:fppfdual}
Let $\cF$ and $\cF'$ be coherent sheaves on $T$, and suppose that for all $\psi:S\to T$ we are given isomorphisms
\[
    \Hom_S(\psi^*\cF,\cO_{S})\simeq\Hom_S(\psi^*\cF',\cO_{S})
\]
compatible with pullbacks. Then $\cF\simeq\cF'$.
\end{lemma}
\begin{proof}
For a coherent sheaf $\cM$ on $T$, consider the $\cO_T$-algebra $\cO_T\oplus\cM$ equipped with the product
\[
    (s,t)(s',t')=(ss',st'+s't).
\]
Let $I(\cM)$ be the spectrum of this algebra. We have obvious morphisms
\[T\xrightarrow{\epsilon}I(\cM)\xrightarrow{\eta}T.\]
Now use
the identity
\[
   \Hom_T(\cF,\cM)=\mathop{\mathrm{Ker}}(\Hom_{I(\cM)}(\eta^*\cF,\cO_{I(\cM)})\xrightarrow{\epsilon^*}\Hom_T(\cF,\cO_T))
\]
and the Yoneda Lemma.
\end{proof}

\subsection{Hamiltonian vector fields}
Let us return to the proof of Theorem~\ref{th:IntegrableSystems}. Consider the composition
\[
    \pi_X^*\Omega_B\to\Omega_X\xrightarrow{\simeq}\cT(X),
\]
where the first morphism is the pullback, and the second one is the isomorphism given by $\omega$. Its image is contained in the
subsheaf $\cT(X/B)\subset\cT(X)$; this is checked by restricting to the smooth locus of $\pi_X$, where it follows because the fibers of $\pi_X$ are Lagrangian. By adjunction, we obtain a morphism of quasicoherent sheaves
\[h:\Omega_B\to\pi_{X,*}\cT(X/B).\]
Note that $h$ is a morphism of $\cO_B$-Lie algebras, where $\Omega_B$ is viewed as an abelian Lie algebra (again, this can be verified on
the smooth locus of $\pi_X$).

The morphism $h$ induces a morphism of sheaves of $\cO_B$-Lie algebras on the big fppf site of $B$:
\[(\Omega_B)_{fppf}\to(\pi_{X,*}\cT(X/B))_{fppf}\to\pi_{X,*}\cT_{fppf}(X/B).\]
We denote the composition by $h_{fppf}$.

\begin{lemma}\label{lm:centralizers of hamiltonians}
\stepzero\noindstep The morphism $h_{fppf}$ is injective;

\noindstep $h_{fppf}((\Omega_B)_{fppf})\subset\pi_{X,*}\cT_{fppf}(X/B)$ is equal to its sheaf of centralizers.
\end{lemma}
\begin{proof}
Let $X^{sm}\subset X$ be the smooth locus of $\pi_X$, and put $\pi_{X^{sm}}=\pi_X|_{X^{sm}}$. Note that the Hamiltonian
vector fields act formally freely transitively along the fibers of $\pi_{X^{sm}}$; that is, the morphism
\[
    \pi_{X^{sm}}^*\Omega_B\to\cT(X^{sm}/B)
\]
is an isomorphism. Moreover, the fibers of $\pi_{X^{sm}}$ are connected. Therefore, the composition of $h_{fppf}$ and the restriction to the smooth locus
\[
    (\Omega_B)_{fppf}\to\pi_{X^{sm},*}\cT_{fppf}(X^{sm}/B)
\]
is injective, and its image is its own centralizer.

It remains to notice that the natural morphism
\[
    \pi_{X,*}\cT_{fppf}(X/B)\to\pi_{X^{sm},*}\cT_{fppf}(X^{sm}/B)
\]
is injective, because the smooth loci are dense in the fibers of $\pi_X$.
\end{proof}

\subsection{The Liouville Theorem} Let us prove the first statement of Theorem~\ref{th:IntegrableSystems}.

Recall from~\cite[Sect.~4c]{GrothendieckSchemasHilbert} that $\Aut_B(X)=\Iso_B(X,X)$ is the scheme representing the functor sending a $B$-scheme $S$ to the set of $S$-automorphisms of $X_S=X\times_BS$. Clearly, $\Aut_B(X)$ is a $B$-group scheme. It follows from \emph{loc.~cit.\/} that it is locally of finite type over $k$. Consider its subgroup scheme $A'$ of automorphisms commuting with the
image of $h_{fppf}$. Thus, $A'(S)$ is the set of $S$-automorphisms of~$X_S$ commuting with the image of $h_{fppf}(S)$.
It is easy to check that $A'$ is a closed subgroup scheme of $\Aut_B(X)$.

\begin{proposition}\label{pr:Lie}
The sheaf $\Lie(A'/B)$ is a sheaf of abelian Lie algebras and
\[
    \Lie(A'/B)\simeq(\Omega_B)_{fppf}.
\]
\end{proposition}
\begin{proof}
It follows from definitions that
\begin{multline*}
    \Lie(\Aut_B(X)/B)(S)=\{\psi\in Mor_S(X_S\times\spec\Z[\epsilon]/\epsilon^2,X_S):\psi|_{X_S}=\Id_{X_S}\}\\
    =(\cT_{fppf}(X/B))(X_S)=(\pi_{X,*}\cT_{fppf}(X/B))(S).
\end{multline*}
Thus $\Lie(\Aut_B(X)/B)=\pi_{X,*}\cT_{fppf}(X/B)$.

It follows directly from definitions that $\Lie(A'/B)\subset\Lie(\Aut_B(X)/B)$ identifies with the sheaf of Lie centralizers
of $h_{fppf}((\Omega_B)_{fppf})$ in $\pi_{X,*}\cT_{fppf}(X/B)$. Now the proposition follows
from Lemma~\ref{lm:centralizers of hamiltonians}.
\end{proof}

Let $A$ be the functor sending a $B$-scheme $S$ to the set of $\psi\in A'(S)$ such that for all geometric points $s\to S$ we have $\psi(s)\in (A'_s)^0$, where $(A'_s)^0$ is the neutral connected component of the $k(s)$-group $A'_s$ (see~\cite[Exp.~$\mathrm{VI_B}$, Def.~3.1]{SGA3-1}).

\begin{corollary}
    The functor $A$ is represented by an open subgroup scheme of $A'$. Moreover, $A$ is commutative and smooth over $B$.
\end{corollary}
\begin{proof}
By Propositions~\ref{pr:Lie} and~\ref{pr:SmoothGrSch}, the group scheme $A'$ is smooth along the unit section. Thus
representability and smoothness of $A$ follow from~\cite[Exp.~$\mathrm{VI_B}$, Thm.~3.10]{SGA3-1}. Next, $A$ is commutative because $\Lie(A/B)=\Omega_B$ is abelian by Proposition~\ref{pr:Lie} and the fibers of the projection $A\to B$ are connected.
\end{proof}

We have a tautological action $\mu:A\times_BX\to X$, which clearly preserves the smooth locus $X^{sm}$ of $\pi_X$.

\begin{proposition}\label{pr:torsor}
$X^{sm}$ is an $A$-torsor.
\end{proposition}
\begin{proof}
Since the statement is \'etale local over $B$, we may assume that $\pi_X$ has a~section $\zeta_X:B\to X^{sm}$. Set
$\mu_\zeta:=\mu\circ(\Id_A\times\zeta_X):A\to X$. We want to prove that $\mu_\zeta$ is an open embedding whose image is $X^{sm}$. Clearly, the image of $\mu_\zeta$ is contained in~$X^{sm}$.

The tangent morphism of $\mu_\zeta$ along $\zeta_A(B)$ is a morphism of sheaves
\[
    \Lie(A/B)=\Omega_B\to\cT(X^{sm}/B)|_{\zeta_X(B)}.
\]
An inspection of the proof of Proposition~\ref{pr:Lie} shows that this morphism is the isomorphism corresponding to the symplectic form $\omega$. It follows that the orbital morphism $\mu_\zeta$ is \'etale. In particular, $\mu_{\zeta}$ is open, and hence
$\mu_\zeta(A)\subset X$ is dense.

Let $Z:=\mu_\zeta^{-1}(\zeta_X(B))$ be the stabilizer of $\zeta_X$. Since $A$ is a commutative group scheme, we see that $Z$ acts trivially on $\mu_\zeta(A)$ and thus on $X$. Now it follows from the universal property of $\Aut_B(X)$ that $Z=\zeta_A(B)$. Thus $\mu_\zeta$ is an open embedding.

It remains to prove that $\mu_\zeta(A)=X^{sm}$. Let $x\in X^{sm}$ be a point lying over $b\in B$. It follows from the above argument that
the orbit of $x$ is open in $\pi_{X^{sm}}^{-1}(b)$ (because we could choose the \'etale local section $\zeta_X$ so that $\zeta_X(b)=x$).
Since $\pi_{X^{sm}}^{-1}(b)$ is irreducible, it follows that it consists of a single orbit. This implies the statement.
\end{proof}

\begin{corollary}
The generic fiber of $\pi_X$ is a torsor over the generic fiber of $\pi_A$ and the latter is an abelian variety.
\end{corollary}

\subsection{Completely integrable systems as degenerate abelian schemes}
Let us prove the second statement of Theorem~\ref{th:IntegrableSystems}.

Denote the projection $A\to B$ by $\pi_A$ and  let us prove that $(\pi_X,\pi_A,\mu)$ is a~degenerate abelian scheme.
Most of the conditions of Definition~\ref{def:DAS} have already been verified. Clearly, $\pi_X$ is a locally complete intersection,
and in particular Gorenstein. The relative dualizing line bundle is the pullback of a line bundle over $B$, because
\[
    \omega_{X/B}\simeq\omega_X\otimes\pi_X^*(\omega_B^{-1}),
\]
and $\omega_X\simeq\cO_X$, since $X$ is symplectic.

It remains to verify the dimensional estimates~\eqref{eq:condition}. Consider the function
\[
    b\mapsto\delta(A_b):B\to\Z.
\]
It is upper-semicontinuous by Proposition~\ref{pp:delta semicontinuous}.

\begin{proposition}[Ng\^o~\cite{NgoBonn}] For every $k$, we have
\[\codim\{b\in B:\delta(A_b)\ge k\}\ge k.\]
\end{proposition}
\begin{proof}
As in the proof of Proposition~\ref{pp:delta semicontinuous}, put
\[
    Y:=\{x\in X:\dim(\St_x)\ge k\},
\]
where $\St_x$ is the stabilizer of $x$, and  recall that
\[
    \{b\in B:\delta(A_b)\ge k\}=\pi_X(Y)
\]
by Proposition~\ref{pp:delta and st}.

On the other hand, for every point $x\in X$, the infinitesimal action of $A$ at $x$ is a morphism
\[
    \Lie(A/B)_{\pi_X(x)}\to\cT_xX,
\]
whose adjoint is identified with
\[
    d\pi_X(x):\cT_xX\to\cT_{\pi_X(x)}B
\]
by the symplectic form. Therefore,
\[
    Y=\{x\in X:\rk d\pi_X(x)\le g-k\}.
\]
In particular, the rank of the differential of the restriction
\[
    \pi_X|_Y:Y\to B
\]
never exceeds $g-k$, and therefore $\dim(\pi_X(Y))\le g-k$ because $\pi_X|_Y$ must be smooth generically.
\end{proof}

This completes the proof of Theorem~\ref{th:IntegrableSystems}.

\bibliographystyle{alphanum}
\bibliography{PartialFourierMukai}
\end{document}